\documentclass[a4paper,10pt]{amsart}

\usepackage[utf8]{inputenc} 
\usepackage[T1]{fontenc}
\usepackage{mathtools,amssymb,amsthm} 
\usepackage{dsfont}
\usepackage[all]{xy}
\usepackage{todonotes}
\usepackage{xparse} 
\usepackage{microtype}
\usepackage{mathrsfs}
\usepackage{varioref}
\usepackage{hyperref}

\input{macros.tex}

\title[Equivalence between coefficient systems and systems of idempotents]{Equivalence of categories between coefficient systems and systems of idempotents}
\author{Thomas Lanard}
\email{thomas.lanard@univie.ac.at}

\begin{document}

\begin{abstract}
The consistent systems of idempotents of Meyer and Solleveld allow to construct Serre subcategories of $\rep[R]{G}$, the category of smooth representations of a $p$-adic group $G$ with coefficients in $R$. In particular, they were used to construct level 0 decompositions when $R=\Zl$, $\lprime \neq p$, by Dat for $\gl{n}$ and the author for a more general group. Wang proved in the case of $\gl{n}$ that the subcategory associated with a system of idempotents is equivalent to a category of coefficient systems on the Bruhat-Tits building. This result was used by Dat to prove an equivalence between an arbitrary level zero block of $\gl{n}$ and a unipotent block of another group. In this paper, we generalize Wang's equivalence of category to a connected reductive group on a non-archimedean local field.
\end{abstract}

\maketitle

\section*{Introduction}

Let $\kk$ be a non-archimedean local field and $G$ the $\kk$-points of a connected reductive group over $\kk$. Let $p$ be the characteristic of the residue field of $\kk$. Denote by $\bt$ the semi-simple Bruhat-Tits building of $G$. The building is a polysimplicial complex, partially ordered by the order relation $\sigma \leq \tau$ if $\sigma$ is a face of $\tau$.

Let $R$ be a commutative ring in which $p$ is invertible. A coefficient system $\Gamma$ on $\bt$ with coefficients in $R$ is a contravariant functor from the category $(\bt,\leq)$ to the category of $R$-modules. In other words, this is the data of $R$-modules $(V_{\sigma})_{\sigma \in \bt}$ and $R$-morphisms $\varphi_{\tau}^{\sigma}:V_{\sigma} \rightarrow V_{\tau}$ if $\tau \leq \sigma$, such that $\varphi_{\sigma}^{\sigma} = \Id$, and $\varphi_{\omega}^{\tau} \circ \varphi_{\tau}^{\sigma} = \varphi_{\omega}^{\sigma}$ if $\omega \leq \tau \leq \sigma$. A coefficient system is said to be $G$-equivariant if for every $g \in G$ and for every $\sigma \in \bt$ there is an isomorphism $\alpha_{g,\sigma}: V_{\sigma} \rightarrow V_{g\sigma}$ compatible with the $\varphi_{\tau}^{\sigma}$ and such that $\alpha_{1,\sigma}=\Id$, $\alpha_{g,h\sigma} \circ \alpha_{h,\sigma} = \alpha_{gh,\sigma}$. To $\Gamma$ we associate a graduated cellular chain complex $C_{*}(\bt,\Gamma)=\bigoplus_{\sigma \in \bt} V_\sigma$ (see Section \ref{secsystdecoeff} for more details and the definition of the differential). The homology of $C_{*}(\bt,\Gamma)$ is then noted $H_{*}(\bt,\Gamma)$.

\medskip

When $R=\mathbb{C}$, a fundamental example of coefficient systems is given by the work of Schneider and Stuhler \cite{SchneiderStuhler}. Let $\sigma \in \bt$ and let $\para{G}{\sigma}$ be the parahoric subgroup of $G$ at $\sigma$. Schneider and Stuhler then constructed a filtration of $\para{G}{\sigma}$ by compact open subgroups $\para{G}{\sigma} \supseteq U_{\sigma}^{(0)} \supseteq U_{\sigma}^{(1)} \supseteq \cdots \supseteq U_{\sigma}^{(r)} \supseteq \cdots$. Then, they associate to $r\in \mathbb{N}$ and $V$ a smooth $RG$-module, the $G$-equivariant coefficient system $\sigma \mapsto V^{U_{\sigma}^{(r)}}$. If $V$ is genertated by $V^{U_{x}^{(r)}}$ for some vertex $x$, then the chain complex of this coefficient system provides a projective resolution of $V$.

\medskip

Complex representations of $p$-adic groups were first studied. Then, in order to extend the local Langlands programme to representations with coefficients in a field, or a ring, as general as possible, it appeared interesting to study $R$-representations. This was first initiated by Vignéras in \cite{vignerasmod}. She extended the natural projective resolutions of Schneider and Stuhler to representations on vector spaces over fields of characteristic not equal to $p$ in \cite{vignsheves}.

\medskip

Meyer and Solleveld generalize these process in \cite{meyer_resolutions_2010}. Denote by $\mathcal{H}_{R}(G)$ the Hecke algebra of $G$ with coefficients in $R$. Let $\bts$ be the set of vertices in $\bt$, that is the set of dimension 0 polysimplices.  Let $e=(e_{\sigma})_{\sigma\in \bt}$ be a system of idempotents of $\mathcal{H}_{R}(G)$ satisfying properties of consistency (see \cite[Def 2.1]{meyer_resolutions_2010} or Definition \ref{defcoherent}). Then the functor $\sigma \to e_{\sigma}(V)$ provides a $G$-equivariant coefficient system. They then showed that if we denote by $\rep[R]{G}$ the abelian category of smooth $G$ representations with coefficients in $R$ and $\rep[R][e]{G}$ the full subcategory of objects $V$ in $\rep[R]{G}$ such that $V=\sum_{x\in \bts}e_{x}V$, then $\rep[R][e]{G}$ is a Serre subcategory (\cite[Thm 3.1]{meyer_resolutions_2010}). We recall that a Serre subcategory is a full subcategory $\mathfrak{S}$ of an Abelian category $\mathfrak{A}$ such that for every exact sequence $0 \to A \to B \to C$ in $\mathfrak{A}$ we have that $B \in \mathfrak{S}$ if and only if $A \in \mathfrak{S}$ and $C \in \mathfrak{S}$.

\medskip

This result allowed Dat in \cite{dat_equivalences_2014} to reconstruct the depth 0 blocks of $\gl{n}(F)$ for $R=\Zl$. Subsequently, decompositions on $\Zl$ were obtained for $G$ which splits over an unramified extension of $\kk$ in \cite{lanard} and \cite{lanard2}.

\medskip

Let $e$ be a consistent system of idempotents. We also assume that if $x$ is a vertex of $\sigma$ then $e_\sigma \in \hecke{\fix{G}{x}}{R}$, where $\fix{G}{x}$ is the pointwise stabilizer of $x$ in $G$. A $e$-coefficient system is a $G$-equivariant coefficient system such that $\varphi_{x}^{\sigma} : V_{\sigma} \rightarrow V_{x}$ induces an isomorphism $V_{\sigma} \overset{\backsim}{\rightarrow} e_{\sigma}(V_{x})$. We denote by $\coef$ the category of $e$-coefficient systems.

\begin{The*}
Let $\Gamma$ be a $e$-coefficient system on $\bt$. Then, the chain complex $C_{*}(\bt,\Gamma)$ is exact except in degree 0.
\end{The*}

For $V$ a smooth $R$-representation of $G$, let $\Gamma(V)$ be the $e$-coefficient system defined by $V_\sigma=e_\sigma(V)$. We then prove the following theorem

\begin{The*}
The functor
\[\begin{array}{ccc}
 \rep[R][e]{G} & \to & \coef \\
 V & \mapsto & \Gamma(V) \\
\end{array}\]
admits a quasi-inverse $\Gamma \mapsto H_0(\bt,\Gamma)$, so induces an equivalence of categories.
\end{The*}

\begin{Rem*}
\begin{enumerate}
\item In the case of $\gl{n}(F)$, these results were proven by Wang in \cite{Wanga}.
\item The equivalence on $\gl{n}(F)$ was used by Dat in \cite{dat_equivalences_2014} to construct an equivalence of categories between an arbitrary block of $\gl{n}(F)$ and a unipotent block of another reductive $p$-adic group. Thus, we can hope that this theorem can be used to show an equivalence between blocks in a more general context.
\item This theorem applies in particular to all the categories constructed in \cite{lanard} and \cite{lanard2}.
\end{enumerate}
\end{Rem*}

Schneider and Stuhler proved in \cite[Thm. V.1]{SchneiderStuhler} a general equivalence of categories between a category of representations and of coefficient systems. In \cite[Thm. V.1]{SchneiderStuhler} the field of coefficient is $R=\mathbb{C}$. However, the main argument is the existence of coefficient systems $\Gamma(V)$, natural in $V$, that provide resolutions of $V$. For the coefficient systems considered in this article, the resolution properties are proven in \cite[Thm. 2.4]{meyer_resolutions_2010}. Hence the result of Schneider and Stuhler is valid in our context. The main difference is that, in the category of coefficient systems, they need to formally invert morphisms $s$ such that $H_0(\bt,s)$ is an isomorphism. One strategy to prove the theorem wanted here, could be to restrict the functor of Schneider and Stuhler to $\coef$ and prove that the localization is trivial. Proving that is similar to prove the results of this article (for instance, it follows from Theorem \ref{theEsssurj}). Therefore, we will instead follow the strategy of \cite{Wanga}. Moreover, this will also show different techniques to prove acyclicity of chain complexes, which is very important.

\bigskip

This paper is composed of four parts. The first one recalls the definitions of coefficient systems and systems of idempotents. The proof in \cite{Wanga} uses geometric properties of the Bruhat-Tits building specific to $\gl{n}$, in particular, the notion of taut paths. These paths do not work well for a general reductive group, so we introduce in Section \ref{secAdmpaths} a more flexible notion: the admissible paths. They allow us to redefine the local maps in Section \ref{secapplocales}. The goal of this third section is to check that this definition is independent of the path chosen. Once this is done, we can easily show, in the last section, that the "new" local maps satisfy all the properties that we need to follow the rest of the proof of Wang.

\section{Coefficient systems and systems of idempotents}
\label{secsystdecoeff}

Let $\kk$ be a non-archimedean local field and $G$ the $\kk$-points of a connected reductive group over $\kk$. Let $p$ be the characteristic of the residue field of $\kk$.

\bigskip
Denote by $\bt$ the semi-simple Bruhat-Tits building associated with $G$ (cf. \cite{bt1} and \cite{BT}). The building is a polysimplicial complex and we denote by $\bts$ the set of dimension 0 polysimplices, that is the vertices. In the rest of this article we will use Latin letters $x$, $y$, $\cdots$ for the vertices and Greek letters $\sigma$, $\tau$, $\cdots$ for general polysimplices. The building $\bt$ is partially ordered by the partial order $\sigma \leq \tau$ if $\sigma$ is a face of $\tau$. A set of polysimplices $\sigma_{1},\cdots,\sigma_{k}$ is said to be adjacent if there is a polysimplex $\sigma$ such as for all $i \in \{1,\cdots,k\}$, $\sigma_{i} \leq \sigma$. If $\{\sigma_{1},\cdots,\sigma_{k}\}$ is a set of adjacent polysimplices we denote by $[\sigma_{1},\cdots,\sigma_{k}]$ the smallest polysimplex containing $\sigma_{1} \cup \cdots \cup \sigma_{k}$. For $\sigma$, $\tau$ two polysimplices, denote by $H(\sigma,\tau)$ the polysimplicial hull of $\sigma$ and $\tau$, that is, the intersection of all the apartments containing $\sigma \cup \tau$.

\bigskip

Let $R$ be a commutative ring in which $p$ is invertible. A coefficient system $\Gamma$ on $\bt$ with coefficients in $R$ is a contravariant functor from the category $(\bt,\leq)$ to the category of $R$-modules. In other words, this is the data of $R$-modules $(V_{\sigma})_{\sigma \in \bt}$ and $R$-morphisms $\varphi_{\tau}^{\sigma}:V_{\sigma} \rightarrow V_{\tau}$ for all polysimplices $\tau \leq \sigma$ such that: $\varphi_{\sigma}^{\sigma} = \Id$ for all $\sigma \in \bt$, and $\varphi_{\omega}^{\tau} \circ \varphi_{\tau}^{\sigma} = \varphi_{\omega}^{\sigma}$ if $\omega, \tau, \sigma$ are polysimplicies such that $\omega \leq \tau \leq \sigma$.

To associate a cellular chain complex to $\Gamma$, each polysimplex is equipped with an orientation that induces an orientation on each of its faces. For all $\sigma,\tau \in \bt$, we define

\[\epsilon_{\tau\sigma}=
\begin{cases}
1 &\text{ if } \tau \leq \sigma \text{ with compatible orientations} \\
-1 &\text{ if } \tau \leq \sigma \text{ with opposite orientations} \\
0 &\text{ if } \tau \text{ is not a face of } \sigma
\end{cases}
\]

Let $\Sigma$ be a subcomplex of $\bt$. The cellular chain complex $C_{*}(\Sigma,\Gamma)$ on $\Sigma$ with coefficients $\Gamma$ is the $\mathbb{N}$-graduated chain complex
\[C_{*}(\Sigma,\Gamma):=\cdots \overset{\partial}{\to} C_c^{or}(\Sigma_{d},\Gamma) \overset{\partial}{\to} \cdots \overset{\partial}{\to} C_c^{or}(\Sigma_{0},\Gamma)\]
where $C_c^{or}(\Sigma_{d},\Gamma):=\bigoplus_{\sigma \in \Sigma, \dim(\sigma)=d} V_\sigma$ and the differential is given by
\[ \partial((v_\sigma)_{\sigma \in \bt})_{\tau} = \sum_{\sigma \in \Sigma} \epsilon_{\tau\sigma}\varphi_{\sigma}^{\tau}(v_\sigma).\]
The homology of $C_{*}(\Sigma,\Gamma)$ is then denoted $H_{*}(\Sigma,\Gamma)$.

\bigskip

A coefficient system is said to be $G$-equivariant if for all $g \in G$ and $\sigma \in \bt$ there is an isomorphism $\alpha_{g,\sigma}: V_{\sigma} \rightarrow V_{g\sigma}$ compatible with the $\varphi_{\tau}^{\sigma}$ and such that $\alpha_{1,\sigma}=\Id$, $\alpha_{g,h\sigma} \circ \alpha_{h,\sigma} = \alpha_{gh,\sigma}$.

\bigskip

We can construct $G$-equivariant coefficient systems from a smooth $RG$-module $V$ and a consistent system of idempotents. We fix a Haar measure on $G$ and denote by $\mathcal{H}_{R}(G)$ the Hecke algebra of $G$ with coefficients in $R$, that is the algebra of functions from $G$ to $R$ locally constant with compact support.

\begin{Def}
\label{defcoherent}
A system of idempotents $e=(e_{x})_{x\in \bts}$ of $\mathcal{H}_{R}(G)$ is consistent if the following properties are satisfied :
\begin{enumerate}
\item $e_{x}e_{y}=e_{y}e_{x}$ when $x$ and $y$ are adjacent.
\item $e_{x}e_{z}e_{y}=e_{x}e_{y}$ when $z \in H(x,y)$ and $z$ is adjacent to $x$.
\item $e_{gx}=ge_{x}g^{-1}$ for all $x\in \bts$ and $g\in G$.
\end{enumerate}
\end{Def}

If $e$ is a consistent system of idempotents and $V$ is a smooth $RG$-module, then the functor $\sigma \to e_{\sigma}(V)$ provides a $G$-equivariant coefficient system.

\bigskip

Let $e$ be a consistent system of idempotents on $\bt$. We will say that $e$ satisfies the condition \eqref{eqcondidem} if
\begin{equation}
\label{eqcondidem}
\text{For every vertex } x \text{ and every polysimplex } \sigma \text{ containing } x \text{ we have } e_\sigma \in \hecke{\fix{G}{x}}{R}
\tag{$\ast$}
\end{equation}
where $\fix{G}{x}$ designate the pointwise stabilizer $x$.

Then if $e$ satisfies \eqref{eqcondidem} and if $\Gamma$ is a $G$-equivariant coefficient system, the condition \eqref{eqcondidem} ensures that $e_\sigma$ acts on $V_x$. A $e$-coefficient system is then a $G$-equivariant coefficient system such that for every  $x \leq \sigma$ the morphism $\varphi_{x}^{\sigma} : V_{\sigma} \rightarrow V_{x}$ induces an isomorphism $V_{\sigma} \overset{\backsim}{\rightarrow} e_{\sigma}(V_{x})$. We denote by $\coef$ the category of $e$-coefficient systems.

\bigskip

Let $\rep[R]{G}$ be the abelian category of smooth $G$-representations with coefficients in $R$ and $\rep[R][e]{G}$ the full subcategory of objects $V$ in $\rep[R]{G}$ such that $V=\sum_{x\in \bts}e_{x}V$. Note that $\rep[R][e]{G}$ is a Serre subcategory by \cite[Thm 3.1]{meyer_resolutions_2010}.

\bigskip
For $e$ a consistent system of idempotents satisfying the condition \eqref{eqcondidem}, we have just constructed a functor
\[\begin{array}{ccc}
 \rep[R][e]{G} & \to & \coef \\
 V & \mapsto & \Gamma(V) \\
\end{array}\]
where $\Gamma(V)$ is the functor $\sigma \to e_{\sigma}(V)$. We wish to show that this functor induces an equivalence of categories.

\section{Admissible paths}

\label{secAdmpaths}

To define his local maps, Wang uses in \cite{Wanga} the notion of taut paths ("chemins tendus" in French) between two vertices of the building. However, this definition works well for $\gl{n}$ but not for a general reductive group. We can define taut paths for a general building, but properties as in \cite[Lem 2.2.5]{Wanga} are no longer true (the simplicial hull of two vertices is not, in general, the intersection of two cones). To have more flexibility than with taut path we introduce in this section the notion of admissible paths. They will allow, in Section \ref{secapplocales}, to define the local maps for $G$.

\begin{Def}
Let $\tau,\sigma \in \bt$. A sequence of polysimplices $\tau_{0},\cdots,\tau_{n} \in \bt$ is called an \textit{admissible path} from $\tau$ to $\sigma$ if
\begin{enumerate}
\item $\tau_{0}=\tau$ and $\tau_{n}=\sigma$
\item $\forall i \in \{0, \cdots, n-1\}$, $\tau_{i+1} \in H(\tau_{i},\sigma)$
\item $\forall i \in \{0, \cdots, n-1\}$, $\tau_{i} \leq \tau_{i+1}$ or $\tau_{i+1} \leq \tau_{i}$
\end{enumerate}
\end{Def}

\begin{Rem}
\label{remCheminTendu}
In \cite[(2.2.4)]{Wanga}, we have the definition of taut paths ("chemins tendus" in French). We can notice that if $(z_{0},\cdots,z_{m})$ is a taut path, then $(z_{0}, [z_{0},z_{1}], z_{1} ,[z_{1},z_{2}],$  $\cdots ,z_{m-1},[z_{m-1},z_{m}],z_{m})$ is an admissible path.
\end{Rem}

\begin{Lem}
\label{lemCheminsTendus}
Let $\tau_{0},\cdots,\tau_{n} \in \bt$.
\begin{enumerate}
\item If $\tau_{0},\cdots,\tau_{n}$ is an admissible path then, for all $k \in \{0,\cdots,n\}$, $\tau_{k},\cdots,\tau_{n}$ is also an admissible path.
\item Let $k \in \{0,\cdots,n\}$ and assume that $\tau_{0},\cdots,\tau_{k}$ and $\tau_{k},\cdots,\tau_{n}$ are admissible paths. Then $\tau_{0},\cdots,\tau_{n}$ is an admissible path if and only if for all $i \in \{0,\cdots,k\}$, $\tau_{k} \in H(\tau_{i},\tau_{n})$.
\item Let $k,l \in \{0,\cdots,n\}$ with $k \leq l$. Then if $\tau_{0},\cdots,\tau_{n}$ is an admissible path and if $\tau_{k},\cdots,\tau_{l}$ are adjacent then $\tau_{0},\cdots,\tau_{k},[\tau_{k},\cdots,\tau_{l}],\tau_{l},\cdots,\tau_{n}$ is again an admissible path.
\end{enumerate}
\end{Lem}

\begin{proof}
\begin{enumerate}
\item It is obvious.
\item Let us assume that $\tau_{0},\cdots,\tau_{n}$ is an admissible path. Then for any $i$, since $\tau_{i+1} \in H(\tau_{i},\tau_{n})$, we have $H(\tau_{i+1},\tau_{n}) \subseteq H(\tau_{i},\tau_{n})$. So for every $i \leq k$, $\tau_{k} \in H(\tau_{i},\tau_{n})$.

Conversely, let us assume that for every $i \in \{0,\cdots,k\}$, $\tau_{k} \in H(\tau_{i},\tau_{n})$ and let us show that $\tau_{0},\cdots,\tau_{n}$ is an admissible path. The only condition we have to check is that $\tau_{i+1} \in H(\tau_{i},\sigma)$. This is immediate for $i \in \{k, \cdots, n-1\}$ because $\tau_{k},\cdots,\tau_{n}$ is an admissible path. For $i<k$, $\tau_{0}, \cdots, \tau_{k}$ being an admissible path, $\tau_{i+1} \in H(\tau_{i},\tau_{k})$. But $\tau_{k} \in H(\tau_{i},\tau_{n})$ so $H(\tau_{i},\tau_{k}) \subseteq H(\tau_{i},\tau_{n})$ and we have the result.

\item All we have to do is to check the conditions of being an admissible path between $\tau_{k}$ and $[\tau_{k},\cdots,\tau_{l}]$ and between $[\tau_{k},\cdots,\tau_{l}]$ and $\tau_{l}$. First we have $\tau_{k} \leq [\tau_{k},\cdots,\tau_{l}]$ and $\tau_{l} \leq [\tau_{k},\cdots,\tau_{l}]$. Since $\tau_{l} \leq [\tau_{k},\cdots,\tau_{l}]$, $\tau_{l} \in H([\tau_{k},\cdots,\tau_{l}],\tau_{n})$. All that remains to be checked is that $[\tau_{k},\cdots,\tau_{l}] \in H(\tau_{k},\tau_{n})$, that is, for any $i \in \{k,\cdots,l\}$, $\tau_{i} \in H(\tau_{k},\tau_{n})$. Now, this condition is satisfied by 2, which completes the proof.
\end{enumerate}
\end{proof}

\begin{Lem}
\label{lemCheminOmega}
Let $\sigma,\tau,\omega \in \bt$ such that $\omega \in H(\sigma,\tau)$. Then there exists an admissible path $\tau_{0}=\tau,\cdots,\tau_{k}=\omega$ from $\tau$ to $\omega$ such that for all $i$, $\omega \in H(\tau_{i},\sigma)$.

In particular, there is an admissible path $\tau_{0}=\tau,\cdots,\tau_{k}=\omega,\cdots,\tau_{n}=\sigma$ from $\tau$ to $\sigma$ such that $\tau_{0},\cdots,\tau_{k}$ is an admissible path from $\tau$ to $\omega$.
\end{Lem}

\begin{proof}
The first assertion is \cite[Lem. 2.15]{meyer_resolutions_2010} and the second one follows from 2. in Lemma \ref{lemCheminsTendus}.
\end{proof}

\section{The local maps}
\label{secapplocales}

The admissible paths allow us to redefine the local maps of \cite{Wanga}. The main purpose of this section will be to show that the definition given below is independent of the choice of the admissible path. This is the most technical section. Hence, before we start, we will give a sketch of the proof. The first step is to simplify the admissible path chosen, without changing the local map. For instance, if $x$ and $y$ are two adjacent vertices we want to simplify a path like $x,[x,y],x,[x,y],y$ into $x,[x,y],y$. This is done in Lemmas \ref{lemadjacent} and \ref{lemsimpliepsilon} and we prove the properties satisfied by this simpler path in Lemma \ref{lemsimplificationchemin}. We continue our simplification process (mostly for technical reasons in the proof) and we show that we can transform an admissible path from $\tau$ to $\sigma$ so that the second vertex is the unique maximal polysimplex in $H(\tau,\sigma)$ adjacent to $\tau$ (Lemma \ref{lemfactorisationdebut}). In this way, we can assume that $\tau$ is maximal. Now, we prove the result by induction. Since we start with a maximal polysimplex, the "size" of the path from the second vertex will be smaller and we can apply the induction hypothesis. Hence we can modify the path from the second vertex, and by choosing a suitable vertex in the middle we can again change the beginning of the path (this is done in Lemma \ref{lemrecurrencechemin}). Hence we can "transform" a path into a "relatively close" one. By repeatedly applying these transformations we show in Proposition \ref{proIndependanceChemin} that we can obtain any other admissible path from $\tau$ to $\omega$.

\bigskip

Let $e$ be a consistent system of idempotents which satisfies \eqref{eqcondidem}. Let $\Gamma=(V_{\sigma})_{\sigma \in \bt}$ be a $e$-coefficient system. When $x \leq \sigma$, we identify $V_{\sigma}$ to $e_{\sigma}(V_{x})$ via the morphism $\varphi_{x}^{\sigma}$. For two polysimplices $\tau,\sigma$ such that $\tau \leq \sigma$, we have $V_{\sigma}=e_{\sigma}(V_{\tau})$ and we denote by $p_{\sigma}^{\tau}$ the projection $V_{\tau} \twoheadrightarrow e_{\sigma}(V_{\tau})=V_{\sigma}$.

Let $\tau,\sigma \in \bt$ such that $\tau \leq \sigma$. Then we get two maps
\[ p_{\sigma}^{\tau} : V_{\tau} \twoheadrightarrow e_{\sigma}V_{\tau}=V_{\sigma},\]
\[ \varphi_{\tau}^{\sigma} : V_{\sigma} \hookrightarrow V_{\tau}.\]
Thus we define if $\tau \leq \sigma$ or $\sigma \leq \tau$
\[ \varepsilon_{\sigma}^{\tau} : V_{\tau} \rightarrow V_{\sigma}\]
by $\varepsilon_{\sigma}^{\tau}=p_{\sigma}^{\tau}$ if $\tau \leq \sigma$ and $\varepsilon_{\sigma}^{\tau}=\varphi_{\sigma}^{\tau}$ if $\tau \geq \sigma$.

\begin{Def}
Let $\tau,\sigma \in \bt$. Take $\tau_{0}=\tau,\cdots,\tau_{n}=\sigma$ an admissible path between $\tau$ and $\sigma$. We define $\varepsilon_{\sigma}^{\tau} : V_{\tau} \rightarrow V_{\sigma}$ by
\[ \varepsilon_{\sigma}^{\tau} = \varepsilon^{\tau_{n-1}}_{\tau_{n}} \circ \cdots \circ \varepsilon_{\tau_{1}}^{\tau_{0}}.\]
\end{Def}

\begin{Rem}
If $(z_{0},\cdots,z_{m})$ is a taut path according to \cite{Wanga}, then the definition of $\varepsilon_{z_{m}}^{z_{0}}$ given in \cite{Wanga} Section 2.2 is the same as the one given here, if we consider the admissible path $(z_{0}, [z_{0},z_{1}], z_{1} ,\cdots,z_{m-1},[z_{m-1},z_{m}],z_{m})$.
\end{Rem}

The purpose of what follows is to show that this definition is independent of the admissible path chosen. We need preliminary lemmas.

\begin{Lem}
\label{lemadjacent}
Let $\tau_{0},\cdots,\tau_{n} \in \bt$ be polysimplices such that for all $i$ we have $\tau_{i} \leq \tau_{i+1}$ or $\tau_{i+1} \leq \tau_{i}$. We also assume that $\tau_{0},\cdots,\tau_{n}$ are adjacent. Then
\[ \varepsilon^{\tau_{n-1}}_{\tau_{n}} \circ \cdots \circ \varepsilon_{\tau_{1}}^{\tau_{0}} = \varepsilon^{[\tau_{0},\cdots,\tau_{n}]}_{\tau_{n}} \circ \varepsilon_{[\tau_{0},\cdots,\tau_{n}]}^{\tau_{0}}.\]
\end{Lem}

\begin{proof}
Let us show the result by induction on $n$.

The result is clear for $n=1$. So let us take $\tau_{0},\cdots,\tau_{n}$ as in the statement and assume the result true for $n-1$. The path $\tau_{1},\cdots,\tau_{n}$ verifies the conditions of the statement thus by the induction hypothesis
\[ \varepsilon^{\tau_{n-1}}_{\tau_{n}} \circ \cdots \circ \varepsilon_{\tau_{2}}^{\tau_{1}} = \varepsilon^{[\tau_{1},\cdots,\tau_{n}]}_{\tau_{n}} \circ \varepsilon_{[\tau_{1},\cdots,\tau_{n}]}^{\tau_{1}}\]
and so
\[\varepsilon^{\tau_{n-1}}_{\tau_{n}} \circ \cdots \circ \varepsilon_{\tau_{1}}^{\tau_{0}} = \varepsilon^{[\tau_{1},\cdots,\tau_{n}]}_{\tau_{n}} \circ \varepsilon_{[\tau_{1},\cdots,\tau_{n}]}^{\tau_{1}} \circ \varepsilon_{\tau_{1}}^{\tau_{0}}.\]

There are two cases:
\begin{enumerate}
\item If $\tau_{0} \leq \tau_{1}$, then $\varepsilon^{\tau_{n-1}}_{\tau_{n}} \circ \cdots \circ \varepsilon_{\tau_{1}}^{\tau_{0}}$ corresponds to the map
\[ V_{\tau_{0}} \twoheadrightarrow e_{\tau_{1}}(V_{\tau_{0}})=V_{\tau_{1}} \twoheadrightarrow e_{[\tau_{1},\cdots,\tau_{n}]}(V_{\tau_{1}})=e_{[\tau_{1},\cdots,\tau_{n}]}( e_{\tau_{1}}(V_{\tau_{0}}))=V_{[\tau_{1},\cdots,\tau_{n}]} \hookrightarrow V_{\tau_{n}}.\]
But $e_{[\tau_{1},\cdots,\tau_{n}]} e_{\tau_{1}}=e_{[\tau_{1},\cdots,\tau_{n}]}$ and since $\tau_{0} \leq \tau_{1}$, $[\tau_{1},\cdots,\tau_{n}]=[\tau_{0},\cdots,\tau_{n}]$ and we have the result.
\item If $\tau_{0} \geq \tau_{1}$, then $\varepsilon^{\tau_{n-1}}_{\tau_{n}} \circ \cdots \circ \varepsilon_{\tau_{1}}^{\tau_{0}}$ corresponds to the map
\[ V_{\tau_{0}} \hookrightarrow e_{\tau_{0}}(V_{\tau_{1}}) \twoheadrightarrow e_{[\tau_{1},\cdots,\tau_{n}]}( e_{\tau_{0}}(V_{\tau_{1}}))\hookrightarrow V_{\tau_{n}}.\]
But
\[e_{[\tau_{1},\cdots,\tau_{n}]}( e_{\tau_{0}}(V_{\tau_{1}}))=e_{[\tau_{1},\cdots,\tau_{n}]} e_{\tau_{0}}(V_{\tau_{0}})\]
because $V_{\tau_{0}}=e_{\tau_{0}}(V_{\tau_{1}})$ and $e_{[\tau_{1},\cdots,\tau_{n}]} e_{\tau_{0}} = e_{[[\tau_{1},\cdots,\tau_{n}],\tau_{0}]} = e_{[\tau_{0},\cdots,\tau_{n}]}$. So the first two arrows correspond to the map
\[V_{\tau_{0}} \twoheadrightarrow e_{[\tau_{0},\cdots,\tau_{n}]}( V_{\tau_{0}})\]
and we have the result.
\end{enumerate}
\end{proof}

\begin{Cor}
\label{coradjacent}
Let $\tau$ and $\sigma$ be adjacent polysimplices. Then $ \varepsilon_{\sigma}^{\tau}$ does not depend on the admissible path chosen.
\end{Cor}

\begin{proof}
Let $\tau_{0},\cdots,\tau_{n}$ be an admissible path between $\tau$ and $\sigma$. The polysimplices $\tau$ and $\sigma$ being adjacent, $H(\sigma,\tau)$ is the closure of $[\sigma,\tau]$. The path $\tau_{0},\cdots,\tau_{n}$ being admissible, for all $i \in \{0,\cdots,n\}$, we have $\tau_i \in H(\sigma,\tau)$ and so $\tau_i \leq [\sigma,\tau]$. Therefore the polysimplices  $\tau_{0},\cdots,\tau_{n}$ are adjacent and $[\tau_{0},\cdots,\tau_{n}] \leq [\sigma,\tau]$. Since $\sigma,\tau \leq [\tau_{0},\cdots,\tau_{n}]$, we have $[\sigma,\tau] \leq [\tau_{0},\cdots,\tau_{n}] $ and so $[\tau_{0},\cdots,\tau_{n}] = [\sigma,\tau]$. Lemma \ref{lemadjacent} tells us that $\varepsilon_{\sigma}^{\tau} = \varepsilon^{[\sigma,\tau]}_{\sigma} \circ \varepsilon_{[\sigma,\tau]}^{\tau}$ which is independent of the chosen path.
\end{proof}

\begin{Lem}
\label{lemsimpliepsilon}
Let $\tau_{0},\cdots,\tau_{n} \in \bt$ and $k,l \in \{0,\cdots,n\}$ with $k \leq l$. Let us assume that $\tau_{0},\cdots,\tau_{n}$ is an admissible path and that $\tau_{k},\cdots,\tau_{l}$ are adjacent. Then
\[\varepsilon^{\tau_{n-1}}_{\tau_{n}} \circ \cdots \circ \varepsilon_{\tau_{1}}^{\tau_{0}}=\varepsilon^{\tau_{n-1}}_{\tau_{n}} \circ \cdots \varepsilon^{\tau_{l}}_{\tau_{l+1}} \circ \varepsilon^{[\tau_{k},\cdots,\tau_{l}]}_{\tau_{l}} \circ\varepsilon^{\tau_{k}}_{[\tau_{k},\cdots,\tau_{l}]} \circ \varepsilon^{\tau_{k-1}}_{\tau_{k}} \circ \cdots \circ \varepsilon_{\tau_{1}}^{\tau_{0}}.\]
\end{Lem}

\begin{proof}
Lemma \ref{lemCheminsTendus} 3. ensures that $\tau_{0},\cdots,\tau_{k},[\tau_{k},\cdots,\tau_{l}],\tau_{l},\cdots,\tau_{n}$ is an admissible path. Now
\[\varepsilon^{\tau_{n-1}}_{\tau_{n}} \circ \cdots \circ \varepsilon_{\tau_{1}}^{\tau_{0}}=\varepsilon^{\tau_{n-1}}_{\tau_{n}} \circ \cdots \varepsilon^{\tau_{l}}_{\tau_{l+1}} \circ \varepsilon^{\tau_{l-1}}_{\tau_{l}} \circ \cdots \circ \varepsilon^{\tau_{k}}_{\tau_{k+1}} \circ \varepsilon^{\tau_{k-1}}_{\tau_{k}} \circ \cdots \circ \varepsilon_{\tau_{1}}^{\tau_{0}}.\]
Since $\tau_{k},\cdots,\tau_{l}$ are adjacent, Lemma \ref{lemadjacent} gives us
\[\varepsilon^{\tau_{l-1}}_{\tau_{l}} \circ \cdots \circ \varepsilon^{\tau_{k}}_{\tau_{k+1}} = \varepsilon^{[\tau_{k},\cdots,\tau_{l}]}_{\tau_{l}} \circ\varepsilon^{\tau_{k}}_{[\tau_{k},\cdots,\tau_{l}]}.\]
And we get the wanted result.
\end{proof}

\begin{Def}
Let $\tau,\sigma \in \bt$. We define $\rho(\tau,\sigma)$ to be the number of polysimplices in $H(\tau,\sigma)$.{}
\end{Def}

\begin{Lem}
\label{lemsimplificationchemin}
Let $\tau_{0},\cdots,\tau_{n}$ be an admissible path and $\omega \in \bt$ such that for all $i$, $\tau_{n} \in H(\tau_{i},\omega)$. Then there exists another admissible path $\tau_{0}',\cdots,\tau_{m}'$ such that
\begin{enumerate}
\item $\tau_{0}'=\tau_{0}$ and $\tau_{m}'=\tau_{n}$
\item $\varepsilon^{\tau_{n-1}}_{\tau_{n}} \circ \cdots \circ \varepsilon_{\tau_{1}}^{\tau_{0}} = \varepsilon^{\tau_{m-1}'}_{\tau_{m}'} \circ \cdots \circ \varepsilon_{\tau_{1}'}^{\tau_{0}'}$
\item if $\tau_{i}'>\tau_{i+1}'$ and $i+1 < m$ then $\rho(\tau_{i+1}',\omega)<\rho(\tau_{i}',\omega)$
\item For all $i$, $\tau_{m}' \in H(\tau_{i}',\omega)$
\end{enumerate}
\end{Lem}

\begin{proof}
Let $k_{1}=\max \{ i>0 | \tau_{0},\cdots,\tau_{i} \text{ are adjacent}\}$ (this set is non empty because $\tau_{1}$ and $\tau_{0}$ are adjacent). Likewise, if $k_{i}$ is defined and different from $n$, we define $k_{i+1}=\max \{ j>k_{i} | \tau_{k_{i}},\cdots,\tau_{j} \text{ are adjacent}\}$. So we have just defined $m'$ integers, $k_{1},\cdots,k_{m'}$ with $k_{m'}=n$. Let $k_{0}=0$ and $m=2m'$.

Let us define for $i\in \{0,\cdots,m'\}$, $\tau_{2i}'=\tau_{k_{i}}$, and for $i\in \{0,\cdots,m'-1\}$, $\tau_{2i+1}'=[\tau_{k_{i}},\cdots,\tau_{k_{i+1}}]$.

Lemma \ref{lemCheminsTendus} 3. applied $m'$ times ensures that $\tau_{0}',\cdots,\tau_{m}'$ is an admissible path. Let us verify that this path satisfies the requested properties.

\begin{enumerate}
\item It is clear that $\tau_{0}'=\tau_{0}$ and $\tau_{m}'=\tau_{n}$.

\item Lemma \ref{lemsimpliepsilon} applied $m'$ times gives us the result.

\item 
We can only have $\tau_{i}'>\tau_{i+1}'$ if $i=2l+1$ is odd. So all we have to do is to show that $H(\tau_{2l+2}',\omega)\varsubsetneq H(\tau_{2l+1}',\omega)$.

By definition of $k_{l+1}$, $\tau_{k_{l+1}+1}$ and $\tau_{2l+1}'$ are not adjacent ($\tau_{k_{l+1}+1}$ exists because by hypothesis $i+1<m$). Take $\mathcal{A}$ an apartment containing $H(\tau_{k_{l}},\omega)$. By hypothesis $\tau_{n} \in H(\tau_{k_{l}},\omega)$ and so $H(\tau_{k_{l}},\tau_{n}) \subseteq H(\tau_{k_{l}},\omega)$. Hence $\tau_{k_{l}}, \cdots, \tau_{n} \in \mathcal{A}$. Since $\tau_{k_{l+1}+1}$ and $\tau_{2l+1}'$ are not adjacent, there exists an affine root $a$ which separates $\tau_{k_{l+1}+1}$ and $\tau_{2l+1}'$. Let us say for example that $a_{|\tau_{k_{l+1}+1}}>0$ and $a_{|\tau_{2l+1}'}<0$. Notice that $\tau_{k_{l+1}+1} \in H(\tau_{k_{l+1}},\tau_{n})\subseteq H(\tau_{k_{l+1}},\omega)=H(\tau_{2l+2}',\omega) \subseteq H(\tau_{2l+1}',\omega)$. Hence if $a_{|\omega}<0$, since $\tau_{k_{l+1}+1} \in H(\tau_{2l+1}',\omega)$ we would get that $a_{|\tau_{k_{l+1}+1}}\leq 0$ which is not. So $a_{|\omega}\geq 0$ and $a$ is an affine root which separates $\tau_{2l+1}'$ from $H(\tau_{k_{l+1}+1},\omega)$ and so $\tau_{2l+1}' \notin H(\tau_{k_{l+1}+1},\omega)$. But $\tau_{k_{l+1}+1}\leq \tau_{k_{l+1}}$ or $\tau_{k_{l+1}+1}\geq \tau_{k_{l+1}}$. We cannot have $\tau_{k_{l+1}+1}\leq \tau_{k_{l+1}}$ because if not $\tau_{k_{l}},\cdots,\tau_{k_{l+1}+1}$ would be adjacent. So $\tau_{k_{l+1}+1}\geq \tau_{k_{l+1}}$. Thus $H(\tau_{k_{l+1}+1},\omega) \supseteq H(\tau_{k_{l+1}},\omega)=H(\tau_{2l+2}',\omega)$. So $\tau_{2l+1}' \notin H(\tau_{2l+2}',\omega)$ which completes the proof.

\item We wish to show that for every $i$, $\tau_{n} \in H(\tau_{i}',\omega)$. If $i=2l$ is even, the result is immediate since $\tau_{2l}'=\tau_{k_{l}}$. We are left with the case $i=2l+1$ odd, that is, we want to show that $\tau_{n} \in H([\tau_{k_{l}},\cdots,\tau_{k_{l+1}}],\omega)$. But $\tau_{k_{l}} \leq [\tau_{k_{l}},\cdots,\tau_{k_{l+1}}]$ so $H(\tau_{k_{l}},\omega) \subseteq H([\tau_{k_{l}},\cdots,\tau_{k_{l+1}}],\omega)$ and $\tau_{n} \in H([\tau_{k_{l}},\cdots,\tau_{k_{l+1}}],\omega)$.
\end{enumerate}
\end{proof}

\begin{Lem}
\label{lemmaxi}
Let $\tau,\sigma \in \bt$. Then there exists a unique maximal polysimplex $\overline{\tau} \in H(\tau,\sigma)$ with $\tau \leq \overline{\tau}$. That is, a polysimplex $\omega$ satisfies $\omega \in H(\tau,\sigma)$ and $\tau \leq \omega$ if and only of $\tau \leq \omega \leq \overline{\tau}$.
\end{Lem}

\begin{proof}
The proof is identical to Lemma 2.9 in \cite{meyer_resolutions_2010} by replacing $x$ with $\sigma$.
\end{proof}

\begin{Lem}
\label{lemhyperplan}
Let $\tau,\sigma,\overline{\tau}$ like in Lemma \ref{lemmaxi}. Let $\mathcal{A}$ be an apartment containing $H(\sigma,\tau)$ and $a$ an affine root. Then if $a_{|\overline{\tau}}=0$ we get $a_{|H(\sigma,\tau)}=0$.
\end{Lem}

\begin{proof}
Assume that $a_{|\overline{\tau}}=0$. Denote by $H=\ker(a)$ an affine hyperplane of $\mathcal{A}$. We then need to show that $H(\sigma,\tau) \subseteq H$. Let us prove the result by contradiction and assume that $H(\sigma,\tau) \nsubseteq H$. Take $x \in H(\sigma,\tau) \backslash H$.

Let $\varphi : [0,1] \rightarrow \bt$ be a geodesic between an interior point of $\overline{\tau}$ and $x$. Since $\varphi(0) \in H$ if there were $t\in ]0,1]$ such that $\varphi(t)\in H$, we would have that for every $t \in [0,1]$, $\varphi(t) \in H$ and in particular that $x \in H$ which is not. So $\varphi(]0,1]) \cap H = \emptyset$. Each $\varphi(t)$ belongs to a polysimplex $\tau(t)$ and the map $t \mapsto \tau(t)$ is piecewise constant, so we can chose $t_{0}$ the smallest positive real where we have a jump. Let $\omega = \tau(t_{0}/2)$. We must have $\overline{\tau} \leq \omega$ or $\omega \leq \overline{\tau}$. But $\omega \nsubseteq H$ and $\overline{\tau} \subseteq H$ so $\overline{\tau} < \omega$. But $\varphi(t_{0}/2) \in H(\tau,\sigma)$ so $\omega \in H(\tau,\sigma)$ which contradicts the definition of $\overline{\tau}$.
\end{proof}

\begin{Lem}
\label{lemenveloppemaximal}
Let $\tau,\omega,\sigma \in \bt$ be polysimplices with $\omega \in H(\tau,\sigma)$. Take, as in Lemma \ref{lemmaxi}, $\overline{\tau}$ maximal in $H(\tau,\sigma)$ such that $\tau \leq \overline{\tau}$ and $\overline{\omega}$ maximal in $H(\omega,\sigma)$ such that $\omega \leq \overline{\omega}$. Then $\overline{\tau} \in H(\tau,\overline{\omega})$.
\end{Lem}

\begin{proof}
Denote by $\overline{\tau}'$ the maximal polysimplex of $H(\tau,\overline{\omega})$ such that $\tau \leq \overline{\tau}'$. Then $H(\tau,\overline{\omega}) \subseteq H(\tau,\sigma)$ and so $\overline{\tau}' \leq \overline{\tau}$.

Let us assume that $\overline{\tau}' \neq \overline{\tau}$. Then $\overline{\tau}'$ is a face of $\overline{\tau}$ and if we take $\mathcal{A}$ an apartment containing $H(\tau,\sigma)$ there exists an affine root $a$ such that $a_{|\overline{\tau}'}=0$ and $a_{|\overline{\tau}}>0$. Since $a_{|\overline{\tau}'}=0$ by Lemma \ref{lemhyperplan} $a_{|H(\tau,\overline{\omega})}=0$ and in particular $a_{|\tau}=0$ and $a_{|\overline{\omega}}=0$. Again by Lemma \ref{lemhyperplan}, since $a_{|\overline{\omega}}=0$ we have $a_{|H(\omega,\sigma)}=0$ and so $a_{|\sigma}=0$. Thus $a_{|\tau}=0$ and $a_{|\sigma}=0$ so $a_{|H(\tau,\sigma)}=0$. But $\overline{\tau} \in H(\tau,\sigma)$, so $a_{|\overline{\tau}}=0$ which is absurd.
\end{proof}

\begin{Lem}
\label{lemfactorisationdebut}
Let $\tau,\sigma \in \bt$. Take $\overline{\tau}$ maximal in $H(\tau,\sigma)$ such that $\tau \leq \overline{\tau}$. Let $\tau_{0}=\tau, \cdots, \tau_{n}=\sigma$ an admissible path between $\tau$ and $\sigma$. Then $\tau, \overline{\tau}, \tau_{0}, \cdots, \tau_{n}$ is again an admissible path and
\[\varepsilon^{\tau_{n-1}}_{\tau_{n}} \circ \cdots \circ \varepsilon_{\tau_{1}}^{\tau_{0}}=\varepsilon^{\tau_{n-1}}_{\tau_{n}} \circ \cdots \circ \varepsilon_{\tau_{1}}^{\tau_{0}} \circ \varepsilon_{\tau_{0}}^{\overline{\tau}} \circ \varepsilon_{\overline{\tau}}^{\tau}.\]
\end{Lem}

\begin{proof}
We have that $\tau \leq \overline{\tau}$ and so $H(\tau,\sigma) \subseteq H(\overline{\tau},\sigma)$. Then since $\overline{\tau} \in H(\tau,\sigma)$ we deduce that $H(\overline{\tau},\sigma) \subseteq H(\tau,\sigma)$. Therefore $H(\tau,\sigma) = H(\overline{\tau},\sigma)$ and $\tau, \overline{\tau}, \tau_{0}, \cdots, \tau_{n}$ is an admissible path.

Let us show by induction on $n$ that for every polysimplices $\sigma$ and $\tau$ and for every admissible path $\tau_{0}, \cdots, \tau_{n}$ between $\tau$ and $\sigma$ that $\varepsilon^{\tau_{n-1}}_{\tau_{n}} \circ \cdots \circ \varepsilon_{\tau_{1}}^{\tau_{0}}=\varepsilon^{\tau_{n-1}}_{\tau_{n}} \circ \cdots \circ \varepsilon_{\tau_{1}}^{\tau_{0}} \circ \varepsilon_{\tau_{0}}^{\overline{\tau}} \circ \varepsilon_{\overline{\tau}}^{\tau}$. The case $n=1$ is obvious. So let us assume the result true for $n-1$ and let us prove it for $n$.

Let $\varphi = \varepsilon^{\tau_{n-1}}_{\tau_{n}} \circ \cdots \circ \varepsilon_{\tau_{1}}^{\tau_{0}}$. The map $\varepsilon^{\tau_{n-1}}_{\tau_{n}} \circ \cdots \circ \varepsilon_{\tau_{1}}^{\tau_{0}} \circ \varepsilon_{\tau_{0}}^{\overline{\tau}} \circ \varepsilon_{\overline{\tau}}^{\tau}$ corresponds to the map
\[V_\tau \overset{p^{\tau}_{\overline{\tau}}}{\twoheadrightarrow} V_{\overline{\tau}}  \overset{\varphi^{\overline{\tau}}_{\tau}}{\hookrightarrow} V_\tau  \overset{\varphi}{\to} V_{\sigma}.\]
Let $v \in V_\tau$, we wish to show that $\varphi(v) = \varphi(e_{\overline{\tau}}(v))$. The result we wish to demonstrate is therefore equivalent to saying that if $u \in V_{\tau}$ is such that $e_{\overline{\tau}}(u)=0$ then $\varphi(u)=0$.

Let $\psi=\varepsilon^{\tau_{n-1}}_{\tau_{n}} \circ \cdots \circ \varepsilon_{\tau_{2}}^{\tau_{1}}$. Denote by $\overline{\tau_{1}}$ the maximal polysimplex of $H(\tau_{1},\sigma)$ such that $\tau_{1} \leq \overline{\tau_{1}}$. Then by the induction hypothesis, we know that if $v \in V_{\tau_{1}}$ is such that $e_{\overline{\tau_{1}}}(v)=0$, $\psi(v)=0$. Let $u \in V_{\tau}$ such that $e_{\overline{\tau}}(u)=0$, since $\varphi = \psi \circ \varepsilon_{\tau_{1}}^{\tau_{0}}$, all we have to do is show that $e_{\overline{\tau_{1}}}(\varepsilon_{\tau_{1}}^{\tau_{0}}(u))=0$. There are two cases:
\begin{itemize}

\item $\tau_{0} \leq \tau_{1}$

Then $\tau_{1} \in H(\tau_{0},\sigma)$ and $\tau_{0} \in H(\tau_{1},\sigma)$ so $H(\tau_{0},\sigma)=H(\tau_{1},\sigma)$. Therefore $\overline{\tau_{1}}=\overline{\tau}$. Thus $e_{\overline{\tau_{1}}}(\varepsilon_{\tau_{1}}^{\tau_{0}}(u))=e_{\overline{\tau}}(\varepsilon_{\tau_{1}}^{\tau_{0}}(u))=\varepsilon_{\tau_{1}}^{\tau_{0}}(e_{\overline{\tau}}u)=0$.

\item $\tau_{0} \geq \tau_{1}$

Then $\varepsilon_{\tau_{1}}^{\tau_{0}}=\varphi_{\tau_{1}}^{\tau_{0}}:V_{\tau_{0}} \hookrightarrow V_{\tau_{1}}$ and $e_{\overline{\tau_{1}}}(\varepsilon_{\tau_{1}}^{\tau_{0}}(u))=e_{\overline{\tau_{1}}}e_{\tau}(\varepsilon_{\tau_{1}}^{\tau_{0}}(u))$. But $\overline{\tau} \in H(\tau,\overline{\tau_{1}})$ by Lemme \ref{lemenveloppemaximal}, so Proposition 2.2 (e) of \cite{meyer_resolutions_2010} tells us that $e_{\overline{\tau_{1}}}e_{\tau}=e_{\overline{\tau_{1}}}e_{\overline{\tau}}e_{\tau}=e_{\overline{\tau_{1}}}e_{\tau}e_{\overline{\tau}}$ and therefore $e_{\overline{\tau_{1}}}(\varepsilon_{\tau_{1}}^{\tau_{0}}(u))=e_{\overline{\tau_{1}}}e_{\tau}e_{\overline{\tau}}(\varepsilon_{\tau_{1}}^{\tau_{0}}(u))=e_{\overline{\tau_{1}}}e_{\tau}(\varepsilon_{\tau_{1}}^{\tau_{0}}(e_{\overline{\tau}}u))=0$.
\end{itemize}
\end{proof}

Let $\mathcal{P}(n)$ be the property that for every $\tau,\sigma \in \bt$ such that $\rho(\tau,\sigma) \leq n$ then $\varepsilon_{\sigma}^{\tau}$ does not depend on the choice of the admissible path.

\begin{Lem}
\label{lemrecurrencechemin}
Let $\tau,\sigma \in \bt$ such that $\rho(\tau,\sigma)>1$. Let $\tau_{0},\cdots,\tau_{n}$ and $\tau_{0}',\cdots,\tau_{m}'$ be two admissible paths between $\tau$ and $\sigma$. We assume that
\begin{itemize}
\item $\mathcal{P}(\rho(\tau,\sigma)-1)$ is true
\item $\rho(\tau_{1},\sigma)<\rho(\tau,\sigma)$ and $\rho(\tau_{1}',\sigma)<\rho(\tau,\sigma)$
\item There exists $\omega \in H(\tau_{1},\sigma) \cap H(\tau_{1}',\sigma)$
\item $\rho(\tau,\omega)<\rho(\tau,\sigma)$
\item $\tau_{1} \in H(\tau,\omega)$ and  $\tau_{1}' \in H(\tau,\omega)$
\end{itemize}
Then
\[\varepsilon^{\tau_{n-1}}_{\tau_{n}} \circ \cdots \circ \varepsilon_{\tau_{1}}^{\tau_{0}}=\varepsilon^{\tau_{m-1}'}_{\tau_{m}'} \circ \cdots \circ \varepsilon_{\tau_{1}'}^{\tau_{0}'}.\]
\end{Lem}

\begin{proof}
$\varepsilon^{\tau_{n-1}}_{\tau_{n}} \circ \cdots \circ \varepsilon_{\tau_{1}}^{\tau_{0}}=(\varepsilon^{\tau_{n-1}}_{\tau_{n}} \circ \cdots \circ \varepsilon_{\tau_{2}}^{\tau_{1}})\circ \varepsilon_{\tau_{1}}^{\tau_{0}}$. Since $\rho(\tau_{1},\sigma)<\rho(\tau,\sigma)$ and that $\mathcal{P}(\rho(\tau,\sigma)-1)$ is true, $\varepsilon^{\tau_{n-1}}_{\tau_{n}} \circ \cdots \circ \varepsilon_{\tau_{2}}^{\tau_{1}}=\varepsilon^{\tau_{1}}_{\sigma}$ does not depend on the admissible path. Moreover $\omega \in H(\tau_{1},\sigma)$ hence by Lemma \ref{lemCheminOmega} there exists an admissible path between $\tau_{1}$ and $\sigma$ going through $\omega$ such that the paths from $\tau_{1}$ to $\omega$ and from $\omega$ to $\sigma$ are admissible. Hence $\varepsilon^{\tau_{1}}_{\sigma}=\varepsilon^{\omega}_{\sigma} \circ \varepsilon^{\tau_{1}}_{\omega}$. Thus we have $\varepsilon^{\tau_{n-1}}_{\tau_{n}} \circ \cdots \circ \varepsilon_{\tau_{1}}^{\tau_{0}}=\varepsilon^{\omega}_{\sigma} \circ \varepsilon^{\tau_{1}}_{\omega} \circ \varepsilon_{\tau_{1}}^{\tau_{0}}$. Since $\rho(\tau,\omega)<\rho(\tau,\sigma)$ and that $\mathcal{P}(\rho(\tau,\sigma)-1)$ is true, $\varepsilon^{\tau}_{\omega}$ does not depend on the admissible path. In particular, since $\tau_{1} \in H(\tau,\omega)$, we can complete the path $\tau_{0},\tau_{1}$ in an admissible path between $\tau$ and $\omega$. The path from $\tau_{1}$ to $\omega$ is then also admissible and the same goes for the path $\tau_{0},\tau_{1}$. Thus $\varepsilon^{\tau_{1}}_{\omega} \circ \varepsilon_{\tau_{1}}^{\tau_{0}}=\varepsilon_{\omega}^{\tau_{0}}$. So we have just shown that
\[\varepsilon^{\tau_{n-1}}_{\tau_{n}} \circ \cdots \circ \varepsilon_{\tau_{1}}^{\tau_{0}}=\varepsilon_{\sigma}^{\omega} \circ \varepsilon_{\omega}^{\tau}.\]

But the assumptions about the paths $\tau_{0},\cdots,\tau_{n}$ and $\tau_{0}',\cdots,\tau_{m}'$ are symmetric, so we also have $\varepsilon^{\tau_{m-1}'}_{\tau_{m}'} \circ \cdots \circ \varepsilon_{\tau_{1}'}^{\tau_{0}'}=\varepsilon_{\sigma}^{\omega} \circ \varepsilon_{\omega}^{\tau}$. And finally
\[\varepsilon^{\tau_{n-1}}_{\tau_{n}} \circ \cdots \circ \varepsilon_{\tau_{1}}^{\tau_{0}}=\varepsilon^{\tau_{m-1}'}_{\tau_{m}'} \circ \cdots \circ \varepsilon_{\tau_{1}'}^{\tau_{0}'}.\]
\end{proof}

\begin{Lem}
\label{lemseparationsimplexe}
Let $\tau,\sigma \in \bt$. Take $\mathcal{A}$ an apartment containing $H(\sigma,\tau)$. We assume that there exists an affine root $a$ such that $a_{|\tau}>0$ and $a_{|\sigma}<0$. Then if $\omega \in H(\tau,\sigma)$ is a polysimplex containing an interior point $x$ with $a(x)=0$, we get $\rho(\omega,\sigma) < \rho(\tau,\sigma)$ and $\rho(\tau,\omega)<\rho(\tau,\sigma)$.
\end{Lem}

\begin{proof}
Since the situation is symmetrical in $\tau$ and $\sigma$, all we have to do is show that $\rho(\omega,\sigma) < \rho(\tau,\sigma)$. Since $\omega \in H(\tau,\sigma)$ we have $H(\omega,\sigma) \subseteq H(\tau,\sigma)$ and $\rho(\omega,\sigma) \leq \rho(\tau,\sigma)$. Hence we need to prove that $H(\omega,\sigma) \varsubsetneq H(\tau,\sigma)$.

Since $a(x)=0$ and $x$ is an interior point of $\omega$ we have $a_{|\omega}=0$. Hence $a_{|\omega} \leq 0$, $a_{|\sigma} < 0$ and $a_{|\tau} > 0$ so $a$ separates $\tau$ from $\omega$ and $\sigma$ and therefore $\tau \notin H(\omega,\sigma)$ which ends the proof.
\end{proof}

\begin{Pro}
\label{proIndependanceChemin}
Let $\tau,\sigma \in \bt$. Then $\varepsilon_{\sigma}^{\tau}$ does not depend on the choice of the admissible path.
\end{Pro}

\begin{proof}
We prove the result by induction on $\rho(\tau,\sigma)$. The result is true if $\rho(\tau,\sigma)=1$. Then we assume that $\rho(\tau,\sigma)>1$ and that $\mathcal{P}(\rho(\tau,\sigma)-1)$ is true.

\bigskip

Let us then take two admissible paths $\tau_{0},\cdots,\tau_{n}$ and $\tau_{0}',\cdots,\tau_{m}'$. Lemma \ref{lemfactorisationdebut} allows us to assume that $\tau = \overline{\tau}$ where $\overline{\tau}$ is the maximal polysimplex of $H(\tau,\sigma)$ such that $\tau \leq \overline{\tau}$. By applying Lemma \ref{lemsimplificationchemin} (with $\omega=\sigma$) we can assume that $\tau_{0},\cdots,\tau_{n}$ and $\tau_{0}',\cdots,\tau_{m}'$ verify the conditions of the latter. Moreover, by removing the first simplices if they are equal, we can assume that $\tau_{0} \neq \tau_{1}$ and $\tau_{0}' \neq \tau_{1}'$. In particular we have that $\rho(\tau_{1},\sigma)<\rho(\tau,\sigma)$ and $\rho(\tau_{1}',\sigma)<\rho(\tau,\sigma)$. The case where $\tau$ and $\sigma$ are adjacent is settled by Corollary \ref{coradjacent}, so we assume that $\tau$ and $\sigma$ are not adjacent.

\bigskip

Let $\mathcal{A}$ be an apartment containing $H(\tau,\sigma)$. Since $\tau$ and $\sigma$ are two polysimplices non adjacent, there exists an affine root $a$ such that $a_{|\tau}>0$ and $a_{|\sigma}<0$. Since $\tau$ is maximal, $\tau_{1}<\tau$ and so $a_{|\tau_{1}}\geq 0$. We deduce from $a_{|\tau_{1}}\geq 0$ and $a_{|\sigma}< 0$ that there exists $x \in H(\tau_{1},\sigma)$ such that $a(x)=0$. Likewise, there exists $y \in H(\tau_{1}',\sigma)$ such that $a(y)=0$. Let $\varphi:[0,1] \rightarrow \bt$ be a geodesic between $x$ and $y$. Since $x,y \in H(\tau,\sigma)$, $\varphi([0,1]) \subseteq H(\tau,\sigma)$. Each $\varphi(t)$ is an interior point of a polysimplex $\omega(t)$. The map $t \mapsto \omega(t)$ is piecewise constant, so take $0=t_{0}<t_{2}<\cdots<t_{2k-2}<t_{2k}=1$ the instants where there is a jump of the map $\omega(t)$. Choose also $t_{1},\cdots,t_{2k-1}$ such that $t_{0} < t_{1} < t_{2} < \cdots < t_{2k-1} < t_{2k}$. Let $\omega_{i}=\omega(t_{i})$. Then $\omega_{2i}$ and $\omega_{2i+2}$ must be faces of $\omega_{2i+1}$ and so for all $i$ we get either $\omega_{i} \leq \omega_{i+1}$ or $\omega_{i+1} \leq \omega_{i}$. Now $a(x)=a(y)=0$ so for all $t$, $a(\varphi(t))=0$. Hence $\omega_{i}$ contains an interior point $\varphi(t_{i})$ such that $a(\varphi(t_{i}))=0$. Lemma \ref{lemseparationsimplexe} ensures that $\rho(\omega_{i},\sigma) < \rho(\tau,\sigma)$ and $\rho(\tau,\omega_{i})<\rho(\tau,\sigma)$. To summarize, we have just built a sequence of polysimplices $\omega_{0},\cdots,\omega_{2k}$ such that for all $i$, $\omega_{i} \in H(\tau,\sigma)$, $\rho(\omega_{i},\sigma) < \rho(\tau,\sigma)$, $\rho(\tau,\omega_{i})<\rho(\tau,\sigma)$, $\omega_{i} \leq \omega_{i+1}$ or $\omega_{i+1} \leq \omega_{i}$, $\omega_{0} \in H(\tau_{1},\sigma)$ and $\omega_{2k} \in H(\tau_{1}',\sigma)$.

\bigskip

Let $i \in \{0,1,\cdots,2k\}$. Since $\omega_{i} \in H(\tau,\sigma)$ take, thanks to Lemma \ref{lemCheminOmega}, an admissible path $\tau_{0}^{i}=\tau,\cdots,\tau_{l_{i}}^{i}=\omega_{i}$ from $\tau$ to $\omega_{i}$ such that for all $j$, $\omega_{i} \in H(\tau_{j}^{i},\sigma)$. With Lemma \ref{lemsimplificationchemin} we can assume that the $\tau_{j}^{i}$ satisfy the conditions of the latter. We complete the path $\tau_{0}^{i},\cdots,\tau_{l_{i}}^{i}$ in an admissible path $\tau_{0}^{i},\cdots,\tau_{n_{i}}^{i}$ between $\tau$ and $\sigma$ thanks to Lemma \ref{lemCheminsTendus} 2. We also assume that $\tau_{0}^{i} \neq \tau_{1}^{i}$ by removing the first polysimplices if they are equal. Since $\tau_{0}^{i}=\tau$ is maximal, $\tau_{0}^{i} > \tau_{1}^{i}$. If $\tau_{1}^{i} \neq \omega_{i}$ then condition 3 of Lemma \ref{lemsimplificationchemin} tells us that $\rho(\tau_{1}^{i},\sigma)<\rho(\tau,\sigma)$. And if $\tau_{1}^{i} = \omega_{i}$ then by construction of $\omega_{i}$, $\rho(\tau_{1}^{i},\sigma)=\rho(\omega_{i},\sigma)<\rho(\tau,\sigma)$. We have just constructed an admissible path $\tau_{0}^{i},\cdots,\tau_{n_{i}}^{i}$ between $\tau$ and $\sigma$ such that $\rho(\tau_{1}^{i},\sigma) < \rho(\tau,\sigma)$, $\omega_{i} \in H(\tau_{1}^{i},\sigma)$ et $\tau_{1}^{i} < \tau$. Denote by
\[\varepsilon_i:=\varepsilon^{\tau_{n_{i}-1}^{i}}_{\tau_{n_{i}}^{i}} \circ \cdots \circ \varepsilon_{\tau_{1}^{i}}^{\tau_{0}^{i}}\]
the local map associated to this path.

\bigskip

Let $i \in \{0,1,\cdots,2k-1\}$. We have $\omega_{i} \leq \omega_{i+1}$ or $\omega_{i+1} \leq \omega_{i}$. Denote by $\omega$ the smallest of the two polysimplices so that $\omega \leq \omega_{i}$ and $\omega \leq \omega_{i+1}$. Then $\tau_{0}^{i},\cdots,\tau_{n_{i}}^{i}$ and $\tau_{0}^{i+1},\cdots,\tau_{n_{i+1}}^{i+1}$ satisfy the conditions of Lemma \ref{lemrecurrencechemin}. Indeed, $\rho(\tau_{1}^{i},\sigma) < \rho(\tau,\sigma)$ and $\rho(\tau_{1}^{i+1},\sigma) < \rho(\tau,\sigma)$. Since $\omega_{i} \in H(\tau_{1}^{i},\sigma)$ and $\omega \leq \omega_{i}$, we have that $\omega \in H(\tau_{1}^{i},\sigma)$, and likewise $\omega \in H(\tau_{1}^{i+1},\sigma)$. Then $\rho(\tau,\omega_{i})<\rho(\tau,\sigma)$ and $\rho(\tau,\omega_{i+1})<\rho(\tau,\sigma)$, therefore, since $\omega=\omega_{i}$ or $\omega=\omega_{i+1}$, $\rho(\tau,\omega)<\rho(\tau,\sigma)$. Finally $\tau_{1}^{i} < \tau$ so $\tau_{1}^{i} \in H(\tau,\omega)$ and likewise $\tau_{1}^{i+1} \in H(\tau,\omega)$. Hence Lemma \ref{lemrecurrencechemin} tells us that $\varepsilon_i=\varepsilon_{i+1}$. This being true for every $i \in \{0,\cdots,2k-1\}$ we get $\varepsilon_0=\varepsilon_{2k}$.

\bigskip

The admissible path $\tau_{0},\cdots,\tau_{n}$ and $\tau_{0}^{0},\cdots,\tau_{n_{0}}^{0}$ also satisfy the conditions of \ref{lemrecurrencechemin} with $\omega = \omega_{0}$ so $\varepsilon^{\tau_{n-1}}_{\tau_{n}} \circ \cdots \circ \varepsilon_{\tau_{1}}^{\tau_{0}}=\varepsilon_0$. It is the same with $\tau_{0}',\cdots,\tau_{m}'$ and $\tau_{0}^{2k},\cdots,\tau_{n_{2k}}^{2k}$ by taking $\omega=\omega_{2k}$ so $\varepsilon^{\tau_{m-1}'}_{\tau_{m}'} \circ \cdots \circ \varepsilon_{\tau_{1}'}^{\tau_{0}'}=\varepsilon_{2k}$. We finally find that
\[\varepsilon^{\tau_{n-1}}_{\tau_{n}} \circ \cdots \circ \varepsilon_{\tau_{1}}^{\tau_{0}}=\varepsilon^{\tau_{m-1}'}_{\tau_{m}'} \circ \cdots \circ \varepsilon_{\tau_{1}'}^{\tau_{0}'},\]
which completes the proof.
\end{proof}

Now that we have shown that these local maps are well defined (using admissible paths) it is easy to check that they satisfy some good properties. In particular, we have the following proposition.

\begin{Pro}
\label{proEpsilonCompo}
Let $\tau,\sigma,\omega \in \bt$ such that $\omega \in H(\tau,\sigma)$. Then
\[ \varepsilon^{\tau}_{\sigma}=\varepsilon^{\omega}_{\sigma} \circ \varepsilon^{\tau}_{\omega}.\]
\end{Pro}

\begin{proof}
Proposition \ref{proIndependanceChemin} tells us that $\varepsilon^{\tau}_{\sigma}$ is independent of the chosen admissible path. All we have to do is then take, thanks to Lemma \ref{lemCheminOmega}, an admissible path between $\tau$ and $\sigma$ going though $\omega$ such that the two paths from $\tau$ to $\omega$ and from $\omega$ to $\sigma$ are admissible.
\end{proof}

\begin{Rem}
Proposition \ref{proEpsilonCompo} proves Lemmas (2.2.8), (2.2.9) and (2.2.10) of \cite{Wanga}.
\end{Rem}

\section{Equivalence of categories}

\label{seceqcat}

In this section, we finish the proof of the two theorems stated in the introduction. In the previous section, we have shown how to define the local maps $\varepsilon^{\tau}_{\sigma}$ in the case of any reductive group, and that these maps satisfy the same properties as in \cite{Wanga}. Hence, we can follow Section 2.3 and 2.4 in \cite{Wanga} to get the result. For the reader's convenience, we will recall here the proof. We will refer to \cite{Wanga} and \cite{meyer_resolutions_2010} for full details of the proofs.

\bigskip

Let $\Gamma=(V_\sigma)_{\sigma \in \bt}$ be a $e$-coefficient system. Let $\Sigma$ be a finite convex subcomplex of $\bt$. Fix a vertex $x \in \Sigma$ and denote by $\underline{V_x}:=(\sigma \mapsto V_x, \varphi_x^\sigma=\Id_{V_x})$ the constant coefficient system with value in $V_x$. The local maps $\{ \varepsilon_x^\sigma\}_{\sigma \in \Sigma}$ induce a morphism of coefficient system and so of chain complexes
\[
	\oplus_{\sigma \in \Sigma} \varepsilon_x^\sigma : C_{*}(\Sigma,\Gamma) \to C_{*}(\Sigma,\underline{V_x})
\]
and a morphism on homologies
\[
	p_x^\Sigma := H_0(\oplus \varepsilon_x^\sigma) : H_0(\Sigma,\Gamma) \to H_0(\Sigma, \underline{V_x}).
\]
Let us remark, that since $\Sigma$ is finite convex, it is contractible, thus $C_{*}(\Sigma,\underline{V_x})$ is a resolution of $V_x$. In particular, $H_0(\Sigma, \underline{V_x})=V_x$ and $p_x^\Sigma : H_0(\Sigma,\Gamma) \to V_x$.

\medskip

Let $\Sigma' \subseteq \Sigma$ be a finite convex subcomplex of $\Sigma$. We have a morphism of complexes
\[
	\oplus_{\sigma \in \Sigma'} \Id_{V_\sigma} : C_{*}(\Sigma',\Gamma) \to C_{*}(\Sigma,\Gamma)
\]
and a morphism on homologies
\[
	i_{\Sigma'}^\Sigma := H_0(\oplus \Id_{V_\sigma}) : H_0(\Sigma',\Gamma) \to H_0(\Sigma,\Gamma)
\]

\medskip

From the definition of $p_x^\Sigma$ it follows the next lemma.
\begin{Lem}[{\cite[Lem. 2.3.1]{Wanga}}]
	\label{lemipcompo}
	The composition of the homology morphisms
	\[
		H_0(\{y\},\Gamma)=V_y \overset{i_{y}^\Sigma}{\longrightarrow} H_0(\Sigma,\Gamma) \overset{p_{x}^\Sigma}{\longrightarrow} H_0(\Sigma, \underline{V_x})=V_x
	\]
	is $\varepsilon_x^y$. In particular $p_x^\Sigma \circ i_x^\Sigma=\Id_{V_x}$.
\end{Lem}

Thus we can define an idempotent $e_x^\Sigma$ by
\[
	e_x^\Sigma: H_0(\Sigma,\Gamma) \overset{p_{x}^\Sigma}{\longrightarrow} V_x \overset{i_{x}^\Sigma}{\longrightarrow} H_0(\Sigma,\Gamma).
\]

The properties of $\varepsilon_\sigma^\tau$ proved in Proposition \ref{proEpsilonCompo} allows us to show that the idempotents $e_x^\Sigma$ satisfies the consistency properties of \cite{meyer_resolutions_2010}. Denote by $\Sigma_0$ the subset of $\Sigma$ of vertices.

\begin{Pro}
	\label{prosystconst}
	We have the following properties:
	\begin{enumerate}
		\item If $x,y,z \in \Sigma_0$, $z,x$ are adjacent and $z \in H(x,y)$, then $e_x^\Sigma \circ e_z^\Sigma \circ e_y^\Sigma = e_x^\Sigma \circ e_y^\Sigma$.
		\item If $x,y \in\Sigma_0$ are adjacent, then $e_x^\Sigma \circ e_y^\Sigma =e_y^\Sigma \circ e_x^\Sigma$.	
	\end{enumerate}
\end{Pro}

\begin{proof}
	We follow the proof of \cite[Lem. 2.3.2]{Wanga}. We just need to adapt a little bit the proof of (2), because, for a general building, $[x,y]$ does not necessarily have dimension 1.

	Since $H_0(\Sigma,\Gamma)=\sum_{w \in \Sigma_0} i_w^\Sigma(V_w)$, it is enough to prove that for all $w\in \Sigma_0$, $e_x^\Sigma \circ e_y^\Sigma =e_y^\Sigma \circ e_x^\Sigma$ (resp. $e_x^\Sigma \circ e_z^\Sigma \circ e_y^\Sigma = e_x^\Sigma \circ e_y^\Sigma$) on $i_w^\Sigma(V_w)$.

	\medskip

	Let us start by (1). By Lemma \ref{lemipcompo}, we have
	\[
		e_x^\Sigma \circ e_z^\Sigma \circ e_y^\Sigma \circ i_w^\Sigma = i_x^\Sigma \circ p_x^\Sigma \circ i_z^\Sigma \circ p_z^\Sigma \circ i_y^\Sigma \circ p_y^\Sigma \circ i_w^\Sigma = i_x^\Sigma \circ \varepsilon^z_x \circ \varepsilon^y_z \circ \varepsilon^w_y
	\]
	and
	\[
		e_x^\Sigma \circ e_y^\Sigma \circ i_w^\Sigma = i_x^\Sigma \circ p_x^\Sigma \circ i_y^\Sigma \circ p_y^\Sigma \circ i_w^\Sigma = i_x^\Sigma \circ \varepsilon^y_x \circ  \varepsilon^w_y
	\]
	We obtain the equality with Proposition \ref{proEpsilonCompo}.

	\medskip

	Now, we prove (2). Let $a \in V_w$. We have
	\[
		e_x^\Sigma \circ e_y^\Sigma \circ i_w^\Sigma(a) -e_y^\Sigma \circ e_x^\Sigma \circ i_w^\Sigma (a) = i_x^\Sigma \circ \varepsilon^y_x \circ  \varepsilon^w_y(a)-i_y^\Sigma \circ \varepsilon^x_y \circ  \varepsilon^w_x(a).
	\]
	Denote by $i_x'^\Sigma : V_x \to V_x \oplus V_y \oplus \bigoplus_{s \in \Sigma_0 \setminus \{x,y\}} V_s$ (resp. $i_y'^\Sigma : V_y \to V_x \oplus V_y \oplus \bigoplus_{s \in \Sigma_0 \setminus \{x,y\}} V_s$) be the natural embedding $u \mapsto (u,0,0,\cdots)$ (resp. $v \mapsto (0,v,0,\cdots)$). We have two commutative diagrams
	\begin{gather*}
		\xymatrix{
   		 V_x \ar[r]^-{i_x'^\Sigma} \ar[dr]^-{i_x^\Sigma} & \bigoplus_{s \in \Sigma_0 } V_s \ar@{->>}[d] \\
    	  & H_0(\Sigma,\Gamma)
  		} 
  		\quad \text{ and } \quad
  		\xymatrix{
   		 V_y \ar[r]^-{i_y'^\Sigma} \ar[dr]^-{i_y^\Sigma} & \bigoplus_{s \in \Sigma_0 } V_s \ar@{->>}[d] \\
    	  & H_0(\Sigma,\Gamma)
  		}.
	\end{gather*}
	Thus, we need to show that $i_x'^\Sigma \circ \varepsilon^y_x \circ  \varepsilon^w_y(a)-i_y'^\Sigma \circ \varepsilon^x_y \circ  \varepsilon^w_x(a) \in \partial(\oplus_{\sigma \in \Sigma, \dim(\sigma)=1}V_\sigma)$. Let $z_0:=x,z_1,\cdots,z_r:=y$ be a sequence of vertices of $[x,y]$ such that for all $0 \leq i \leq r-1$, $\dim([z_i,z_{i+1}])=1$, and $r$ is minimal. Now we define an element $b \in \oplus_{\sigma \in \Sigma, \dim(\sigma)=1}V_\sigma$ by $b_{[z_i,z_{i+1}]}:= \varepsilon^{[x,y]}_{[z_i,z_{i+1}]} \circ \varepsilon^{w}_{[x,y]} (a) \in V_{[z_i,z_{i+1}]}$ and 0 elsewhere. By Lemma \ref{proEpsilonCompo}, 
	\begin{align*}
		\varphi^{[x,z_1]}_x(b_{[x,z_1]}) &=\varepsilon^{[x,z_1]}_x  \circ  \varepsilon^{[x,y]}_{[x,z_{1}]} \circ \varepsilon^{w}_{[x,y]} (a) = \varepsilon^{[x,y]}_{x} \circ \varepsilon^{w}_{[x,y]} (a) \\
		&= \varepsilon^{[x,y]}_{x} \circ \varepsilon^{y}_{[x,y]} \circ \varepsilon^{w}_{y} (a) = \varepsilon^{y}_{x} \circ \varepsilon^{w}_{y} (a).
	\end{align*}
	In the same way,
	\[
		\varphi^{[y,z_{r-1}]}_y(b_{[y,z_{r-1}]}) =\varepsilon^{[y,z_{r-1}]}_y  \circ  \varepsilon^{[x,y]}_{[y,z_{r-1}]} \circ \varepsilon^{x}_{[x,y]} \circ \varepsilon^{w}_{x} (a)   = \varepsilon^{x}_{y} \circ \varepsilon^{w}_{x} (a).
	\]
	Also, for $1 \leq i \leq r-1$,
	\[
	 	\varphi^{[z_i,z_{i+1}]}_{z_i}(b_{[z_i,z_{i+1}]}) = \varepsilon^{[z_i,z_{i+1}]}_{z_i} \circ \varepsilon^{[x,y]}_{[z_i,z_{i+1}]} \circ \varepsilon^{w}_{[x,y]} (a) = \varepsilon^{[x,y]}_{z_i} \circ \varepsilon^{w}_{[x,y]} (a)
	 \]
	 \[
	 	\varphi^{[z_i,z_{i-1}]}_{z_{i}}(b_{[z_i,z_{i-1}]}) = \varepsilon^{[z_i,z_{i-1}]}_{z_{i}} \circ \varepsilon^{[x,y]}_{[z_i,z_{i-1}]} \circ \varepsilon^{w}_{[x,y]} (a) = \varepsilon^{[x,y]}_{z_{i}} \circ \varepsilon^{w}_{[x,y]} (a)
	 \]
	 Therefore 
	 \[
	 	\varphi^{[z_i,z_{i+1}]}_{z_i}(b_{[z_i,z_{i+1}]}) - \varphi^{[z_i,z_{i-1}]}_{z_{i}}(b_{[z_i,z_{i-1}]}) = 0.
	 \]

	Thus, $i_x'^\Sigma \circ \varepsilon^y_x \circ  \varepsilon^w_y(a)-i_y'^\Sigma \circ \varepsilon^x_y \circ  \varepsilon^w_x(a) = \partial(b)$ and this finish the proof.
		
\end{proof}

Now, since this system of idempotents satisfies the consistency properties, we can follow the strategy of Meyer and Solleveld. Thanks to property (2), we can define $e_\sigma^\Sigma := \prod_{x \leq \sigma, x \in \Sigma_0} e_x^\Sigma \in \End(H_0(\Sigma,\Gamma))$ and, as in \cite[Thm. 2.12]{meyer_resolutions_2010}, $u_{\Sigma'}^{\Sigma}:=\sum_{\sigma \in \Sigma'} (-1)^{\dim(\sigma)} e_{\sigma}^{\Sigma}$.

\begin{Pro}
	\label{prousig}
	We have
	\begin{align*}
		u_{\Sigma'}^{\Sigma}(H_0(\Sigma,\Gamma))&=\sum_{x\in \Sigma'_0} \im(e_x^{\Sigma})\\
		\Ker(u_{\Sigma'}^{\Sigma})&=\bigcap_{x\in \Sigma'_0} \Ker(e_x^{\Sigma})\\
		H_0(\Sigma,\Gamma)&=u_{\Sigma'}^{\Sigma}(H_0(\Sigma,\Gamma)) \oplus \Ker(u_{\Sigma'}^{\Sigma})
	\end{align*}

	Moreover, if $\Sigma'=\Sigma$, then $\Ker(u_{\Sigma}^{\Sigma})=\bigcap_{x\in \Sigma_0} \Ker(e_x^{\Sigma})=0$.
\end{Pro}

\begin{proof}
	By Proposition \ref{prosystconst}, the system of idempotent $(e_x^\Sigma)_{x \in \Sigma_0}$ satisfies the consistency properties of \cite{meyer_resolutions_2010}. Hence the proof of \cite[Thm. 2.12]{meyer_resolutions_2010} applies here. Now if $\Sigma'=\Sigma$, then $H_0(\Sigma,\Gamma)=\sum_{x\in \Sigma'_0} i_x^{\Sigma}(V_x) = \sum_{x\in \Sigma'_0} \im(e_x^{\Sigma})$ thus $\Ker(u_{\Sigma}^{\Sigma})=0$.
\end{proof}

\begin{Pro}[{\cite[Prop. 2.4.1 (b)]{Wanga}}]
	\label{proinjisigma}
	The map $i_{\Sigma'}^{\Sigma} : H_0(\Sigma',\Gamma) \to H_0(\Sigma,\Gamma)$ is injective.
\end{Pro}

\begin{proof}
	We have a commutative diagram 
	\[
	 \xymatrix{
    H_0(\Sigma',\Gamma) \ar[r]^-{i_\Sigma^{\Sigma'}} \ar[d]^-{p_x^{\Sigma'}} & H_0(\Sigma,\Gamma) \ar[d]^-{p_x^{\Sigma}} \\
     V_x \ar[r]^-{\Id} & V_x
  }
\]
Let $a \in H_0(\Sigma',\Gamma)$ such that $i_\Sigma^{\Sigma'}(a)=0$. By the previous diagram, for all $x \in \Sigma'_0$, $p_x^{\Sigma'}(a)=0$. Thus $e_x^{\Sigma'}(a)=i_x^{\Sigma'} \circ p_x^{\Sigma'}(a)=0$, and $a \in \bigcap_{x\in \Sigma'_0} \Ker(e_x^{\Sigma'})$. By Proposition \ref{prousig}, $a=0$.
\end{proof}

\begin{Rem}
	The proof of \cite[Prop. 2.4.1 (b)]{Wanga} also shows that the image of $i_{\Sigma'}^{\Sigma}$ is $u_{\Sigma'}^{\Sigma}(H_0(\Sigma,\Gamma))$.
\end{Rem}

\begin{Lem}
	\label{lemHnzero}
	Let $\Sigma$ be a finite convex subcomplex. For all $n>0$, $H_n(\Sigma,\Gamma)=0$.
\end{Lem}

\begin{proof}
	As in \cite[Prop. 2.4.1 (a)]{Wanga} we prove the result by induction.

	\begin{Lem}
		Let us assume that $\Sigma$ is not a polysimplex and that for all $\Sigma' \subsetneq \Sigma$ finite convex and for all $n>0$, $H_n(\Sigma',\Gamma)=0$. Then for all $n>0$, $H_n(\Sigma,\Gamma)=0$.
	\end{Lem}

	\begin{proof}
		By \cite[Section 2.5]{meyer_resolutions_2010} we can decompose $\Sigma$ as $\Sigma=\Sigma_{+} \cup \Sigma_{-}$, with $\Sigma_{+}$, $\Sigma_{-}$ and $\Sigma_{0}:-\Sigma_{+}\cap \Sigma_{-}$ are finite convex proper subcomplexes. The cellular chain complexes for these subcomplexes form an exact sequence
		\[
			C_{*}(\Sigma_{0},\Gamma) \rightarrowtail C_{*}(\Sigma_{+},\Gamma) \oplus C_{*}(\Sigma_{-},\Gamma) \twoheadrightarrow C_{*}(\Sigma,\Gamma)
		\]
		which generates a Mayer–Vietoris long exact sequence for their homology groups. Since $\Sigma_{+}$, $\Sigma_{-}$ and $\Sigma_{0}$ satisfies the hypothesis of the Lemma, we get that $H_n(\Sigma,\Gamma)=0$ for $n\geq 2$. By Proposition \ref{proinjisigma} the map $i_{\Sigma_0}^{\Sigma_{+}} : H_0(\Sigma_0,\Gamma) \to H_0(\Sigma_{+},\Gamma)$ is injective, hence $H_1(\Sigma,\Gamma)=0$.
	\end{proof}

	We are left with the case where $\Sigma$ is a single polysimplex. Let $\sigma \subseteq \Sigma$ and $x\in \Sigma_0$. We define an idempotent $e_x^\sigma \in \End(V_\sigma)$ by  $e_x^\sigma:= \varepsilon_{\sigma}^{[x,\sigma]} \circ \varepsilon^{\sigma}_{[x,\sigma]}$. Let $y\in \Sigma_0$ (hence $y$ is adjacent to $x$) and $z\in H(x,y)$. By Lemma \ref{lemadjacent}, $e_x^\sigma e_y^\sigma = \varepsilon_{\sigma}^{[x,y,\sigma]} \circ \varepsilon^{\sigma}_{[x,y,\sigma]} = e_y^\sigma e_x^\sigma$ and $e_x^\sigma e_z^\sigma e_y^\sigma = \varepsilon_{\sigma}^{[x,y,z,\sigma]} \circ \varepsilon^{\sigma}_{[x,y,z,\sigma]}=\varepsilon_{\sigma}^{[x,y,\sigma]} \circ \varepsilon^{\sigma}_{[x,y,\sigma]}=e_x^\sigma e_y^\sigma$. These idempotents satisfies the consistency properties, hence we can follow \cite[Section 2.5]{meyer_resolutions_2010}. For $I \subseteq \Sigma_0$, denote by 
	\[
		e_I^{\sigma,0}:= \prod_{x\in I} e_x^\sigma \prod_{x\notin I} (1-e_x^\sigma) \in \End(V_\sigma)
	\]

Let $\Sigma_I$ be the subcomplex of $\Sigma$ spanned by $I$. As in \cite[Section 2.5]{meyer_resolutions_2010} and \cite[Prop. 2.4.1 (a)]{Wanga} we have
	\begin{enumerate}
		\item $e_I^{\sigma,0}(V_\sigma) = 0 $  if $\sigma_0 \subsetneq I$ 
		\item $\Id_{V_\sigma}=\sum_{I\subseteq \Sigma_0} e_I^{\sigma,0} \in \End(V_\sigma)$
		\item  $ e_I^{\sigma,0}  e_J^{\sigma,0}$, if $I\neq J$
		\item $ e_I^{\sigma,0}=0$ if $\Sigma_I$ is not a single face of $\Sigma$
	\end{enumerate}

	Denote by $e_I^{0} := \oplus_{\sigma}e_I^{\sigma,0} \in \End(C_{*}(\Sigma,\Gamma))$ the endomorphism on the chain complex, it commutes with the differential. Thus we get a decomposition
	\[
		C_{*}(\Sigma,\Gamma) = \bigoplus_{I\subseteq \Sigma_0} e_I^{0}(C_{*}(\Sigma,\Gamma)).
	\]
	If $e_I^{0}(C_{*}(\Sigma,\Gamma))\neq 0$ then $I \neq \emptyset$ and $\Sigma_I$ is a single face of $\Sigma$. The chain complex $ e_I^{0}(C_{*}(\Sigma,\Gamma))$ computes the homology of $\Sigma_I$ with constant coefficients in $e_I^{0}(V_{\Sigma_I})$. Since $\Sigma_I$ is contractible, $ e_I^{0}(C_{*}(\Sigma,\Gamma))$ is a resolution of $e_I^{0}(V_{\Sigma_I})$ and we get that $H_n(\Sigma,\Gamma)=0$ for $n\geq 1$.
\end{proof}

We can now prove the first theorem stated in the introduction.

\begin{The}
Let $e$ be a consistent system of idempotents satisfying condition \eqref{eqcondidem} and $\Gamma$ a $e$-coefficient system on $\bt$. Then the chain complex $C_{*}(\bt,\Gamma)$ is exact except in degree 0.
\end{The}

\begin{proof}
	We follow the proof of \cite[Section 2.5]{meyer_resolutions_2010}. Let $(\Sigma_n)_{n\in \mathbb{Z}}$ be an increasing sequence of  finite convex subcomplexes of $\bt$ such that $\bt=\cup_n \Sigma_n$ and $C_{*}(\bt,\Gamma) = \varinjlim C_{*}(\Sigma_n,\Gamma)$. By Lemma \ref{lemHnzero}, for all $n$ the complex
	\[
		C_{*}(\Sigma_n,\Gamma) \to H_0(\Sigma_n,\Gamma) \to 0
	\]
	is exact. The exactness of inductive limits in the category of $R$-modules tells us that $C_{*}(\bt,\Gamma)$ is a resolution of $\varinjlim H_0(\Sigma_n,\Gamma)$. That is, for all $n\geq 1$, $H_n(\bt,\Gamma)=0$ and $H_0(\bt,\Gamma)=\varinjlim H_0(\Sigma_n,\Gamma)$.
\end{proof}

Let us now move on to the equivalence of categories. Let $x$ be a vertex. If $\Sigma' \subseteq \Sigma$ are two finite convex subcomplexes of $\bt$ containing $x$, then we have a commutative diagram:
\[
	 \xymatrix{
    V_x \ar[r]^-{i_x^{\Sigma'}} \ar[rd]^-{i_x^{\Sigma}} & H_0(\Sigma',\Gamma) \ar[d]^-{i_{\Sigma'}^{\Sigma}} \\
     & H_0(\Sigma,\Gamma)
  }
\]
This induces an injection $i_x : V_x \to H_0(\bt,\Gamma)$.

\begin{Lem}
	\label{lemDiagCommix}

	Let $x$ and $y$ be two adjacent vertices. Then we have a commutative diagram
	\[
	 \xymatrix{
    V_x \ar[r]^-{i_x} \ar[d]^-{\varepsilon_y^x} & H_0(\bt,\Gamma) \ar[d]^-{e_y} \\
    V_y \ar[r]^-{i_y} & H_0(\bt,\Gamma)
  }
\]
\end{Lem}

\begin{proof}
	We follow {\cite[Lem. 2.4.2]{Wanga}} and adapt the proof as in Proposition \ref{prosystconst}. Let $i_x' : V_x \to V_x \oplus V_y \oplus \bigoplus_{s \in \bts \setminus \{x,y\}} V_s$ and $i_y' : V_y \to V_x \oplus V_y \oplus \bigoplus_{s \in \bts \setminus \{x,y\}} V_s$ be the natural embeddings. Let $a \in V_x$. We want to show that $e_y(i'_x(a))-i'_y(\varepsilon^x_y(a)) \in \partial(\bigoplus_{\sigma \in \bt, \dim(\sigma)=1} V_\sigma)$. Let $z_0:=x,z_1,\cdots,z_r:=y$ be a sequence of vertices of $[x,y]$ such that for all $0 \leq i \leq r-1$, $\dim([z_i,z_{i+1}])=1$, and $r$ is minimal. Now we define an element $b \in \oplus_{\sigma \in \bt, \dim(\sigma)=1}V_\sigma$ by $b_{[z_i,z_{i+1}]}:= \varepsilon^{[x,y]}_{[z_i,z_{i+1}]} \circ \varepsilon^{x}_{[x,y]} (a) \in V_{[z_i,z_{i+1}]}$ and 0 elsewhere. We are using repeatedly Lemma \ref{proEpsilonCompo}, 
	\begin{align*}
		\varphi^{[x,z_1]}_x(b_{[x,z_1]}) &=\varepsilon^{[x,z_1]}_x  \circ  \varepsilon^{[x,y]}_{[x,z_{1}]} \circ \varepsilon^{x}_{[x,y]} (a) = \varepsilon^{[x,y]}_{x} \circ \varepsilon^{x}_{[x,y]} (a) \\
		&= e_{[x,y]}(a)=e_ye_x(a)=e_y(a).
			\end{align*}
	In the same way,
	\[
		\varphi^{[y,z_{r-1}]}_y(b_{[y,z_{r-1}]}) =\varepsilon^{[y,z_{r-1}]}_y  \circ  \varepsilon^{[x,y]}_{[y,z_{r-1}]} \circ \varepsilon^{x}_{[x,y]} (a)   = \varepsilon^{x}_{y} (a).
	\]
	Also, for $1 \leq i \leq r-1$,
	\[
	 	\varphi^{[z_i,z_{i+1}]}_{z_i}(b_{[z_i,z_{i+1}]}) = \varepsilon^{[z_i,z_{i+1}]}_{z_i} \circ \varepsilon^{[x,y]}_{[z_i,z_{i+1}]} \circ \varepsilon^{x}_{[x,y]} (a) = \varepsilon^{[x,y]}_{z_i} \circ \varepsilon^{x}_{[x,y]} (a)
	 \]
	 \[
	 	\varphi^{[z_i,z_{i-1}]}_{z_{i}}(b_{[z_i,z_{i-1}]}) = \varepsilon^{[z_i,z_{i-1}]}_{z_{i}} \circ \varepsilon^{[x,y]}_{[z_i,z_{i-1}]} \circ \varepsilon^{x}_{[x,y]} (a) = \varepsilon^{[x,y]}_{z_{i}} \circ \varepsilon^{x}_{[x,y]} (a)
	 \]
	 therefore 
	 \[
	 	\varphi^{[z_i,z_{i+1}]}_{z_i}(b_{[z_i,z_{i+1}]}) - \varphi^{[z_i,z_{i-1}]}_{z_{i}}(b_{[z_i,z_{i-1}]}) = 0.
	 \]
	 Thus $e_y(i'_x(a))-i'_y(\varepsilon^x_y(a)) = \partial(b).$
\end{proof}

\begin{The}
\label{theEsssurj}
Let $e$ be a consistent system of idempotents satisfying condition \eqref{eqcondidem} and $\Gamma$ a $e$-coefficient system on $\bt$. Then $\Gamma$ is isomorphic to the coefficient system $(\sigma \mapsto e_\sigma(H_0(\bt,\Gamma)))$.
\end{The}

\begin{proof}
	The proof is the same as in \cite[Section 2.4]{Wanga}. We recall it here for the reader's convenience. By Lemma \ref{lemDiagCommix}, the morphism $i_x$ has its image in $e_x(H_0(\bt,\Gamma))$. Let $\sigma$ be a polysimplex containing $x$ and $i_\sigma:=i_{x|e_\sigma(V_x)}$. The morphism $i_\sigma$ sends $\varphi^{\sigma}_x(V_\sigma)$ into $e_\sigma(H_0(\bt,\Gamma))$, and is independent of the choice of the vertex $x$ contained in $\sigma$. Thus the morphisms $\{i_\sigma\}_{\sigma \in \bt}$ induce a morphism of coefficient systems
	\[
		\Gamma \to (\sigma \mapsto e_\sigma(H_0(\bt,\Gamma)))_{\sigma \in \bt}.
	\]

	Since $H_0(\bt,\Gamma)= \sum_{y \in \bts} i_y(V_y)$ it is enough to prove that for all $y\in \bts$, $e_\sigma(i_y(V_y)) \subseteq i_x(V_\sigma)$. Let $y \in \bts$. Now take a sequence of vertices $z_0=y, z_1, \cdots, z_m=x$ such that $z_i$ and $z_{i+1}$ are adjacent, and $z_{i+1} \in H(z_i,x)$. Then
	\[
		e_\sigma(i_y(V_y)) = e_\sigma e_y(i_y(V_y)) = e_\sigma e_{z_1}e_y(i_y(V_y)) = \cdots = e_\sigma e_{z_m} \cdots e_{z_1}e_y(i_y(V_y)).
	\]
	By Lemma \ref{lemDiagCommix}
	\begin{align*}
		e_\sigma e_{z_m} \cdots e_{z_1}e_y(i_y(V_y)) &= e_\sigma e_{z_m} \cdots e_{z_2}i_{z_1}(\varepsilon_{z_1}^{y}(V_y)) \subseteq e_\sigma e_{z_m} \cdots e_{z_2}i_{z_1}(V_{z_1}) \\
		&= e_\sigma e_{z_m} \cdots e_{z_2}i_{z_2}(\varepsilon_{z_2}^{z_1}(V_{z_1})) \subseteq e_\sigma e_{z_m} \cdots i_{z_2}(V_{z_2}) \\
		&\subseteq \cdots \subseteq e_\sigma i_x(V_x) = i_x(V_\sigma).
	\end{align*}
\end{proof}

Let $e$ be a system of idempotents and $V$ a smooth $RG$-module. We define a coefficient system $\Gamma(V)$ by $V_\sigma:=e_\sigma V$ and for $\tau \leq \sigma$, $\varphi_{\sigma}^{\tau}$ is the inclusion $V_\sigma \hookrightarrow V_\tau$.

\begin{The}
\label{theeqcat}
Let $e$ be a consistent system of idempotents satisfying the condition \eqref{eqcondidem}. Then the functor
\[\begin{array}{ccc}
 \rep[R][e]{G} & \to & \coef \\
 V & \mapsto & \Gamma(V) \\
\end{array}\]
admits a quasi-inverse $\Gamma \mapsto H_0(\bt,\Gamma)$, hence induces an equivalence of categories.
\end{The}

\begin{proof}
This is done in \cite[Cor (2.1.11)]{Wanga}. By Theorem \ref{theEsssurj}, the functor $V  \mapsto  \Gamma(V)$ is essentially surjective. Hence, we need to prove that it is fully faithful.

For $V,W \in \rep[R][e]{G}$, $\Gamma$ induces a morphism
\[\begin{array}{ccc}
\Gamma :  \Hom_{RG}(V,W) & \to & \Hom_{\Coef_e}(\Gamma(V),\Gamma(W)) \\
 f & \mapsto & (f_\sigma=f_{|e_\sigma(V)} : V_\sigma \to W_\sigma) \\
\end{array}.\]

Let $f \in \Hom_{RG}(V,W)$ such that $\Gamma(f)=0$. Then for all vertex $x$, $f_{|e_x(V)}=0$. Since $V \in \rep[R][e]{G}$, $V=\sum_{x \in \bts} e_x(V)$ and $f=0$. Hence the injectivity of $\Gamma$. For the surjectivity, let $(g_\sigma)_{\sigma \in \bt} \in \Hom_{\Coef_e}(\Gamma(V),\Gamma(W))$. We get a morphism of complexes
\[
	g : C_{*}(\bt,\Gamma(V)) \to C_{*}(\bt,\Gamma(W)).
\]

By \cite[Thm. 2.4]{meyer_resolutions_2010} $C_{*}(\bt,\Gamma(V)) \to V \to 0$ (resp. $C_{*}(\bt,\Gamma(W)) \to W \to 0$) is a resolution of $V$ (resp. W). Hence $H_0(g)$ induces a morphism of $RG$-modules
\[
	H_0(g) : V \to W
\]

By definition, for every $x \in \bts$, we have $H_0(g)_{|e_x(V)}=g_x$. Hence for every $\sigma \in \bt$ containing $x$, $H_0(g)_{|e_\sigma(V)}=(H_0(g)_{|e_x(V)})_{|e_\sigma(V)}=g_{x|e_\sigma(V)}=g_\sigma$. And we have the surjectivity.
\end{proof}

\bibliographystyle{hep}
\bibliography{biblio}

\begin{thebibliography}{{Lan}18b}

\bibitem[BT72]{bt1}
F.~Bruhat and J.~Tits, \textsl{ Groupes r\'{e}ductifs sur un corps local},
\newblock Inst. Hautes \'{E}tudes Sci. Publ. Math. (41), 5--251 (1972).

\bibitem[BT84]{BT}
F.~Bruhat and J.~Tits, \textsl{ Groupes r\'eductifs sur un corps local. {II}.
  {S}ch\'emas en groupes. {E}xistence d'une donn\'ee radicielle valu\'ee},
\newblock Inst. Hautes \'Etudes Sci. Publ. Math. (60), 197--376 (1984).

\bibitem[Dat18]{dat_equivalences_2014}
J.-F. Dat, \textsl{ Equivalences of tame blocks for p-adic linear groups},
\newblock Math. Ann. \textbf{ 371}(1-2), 565--613 (2018).

\bibitem[Lan18a]{lanard}
T.~Lanard, \textsl{ Sur les $\ell$-blocs de niveau zéro des groupes
  $p$-adiques},
\newblock Compositio Mathematica \textbf{ 154}(7), 1473–1507 (2018).

\bibitem[{Lan}18b]{lanard2}
T.~{Lanard}, \textsl{ {Sur les $\ell$-blocs de niveau zéro des groupes
  $p$-adiques II}},
\newblock arXiv e-prints , arXiv:1806.09543 (June 2018), {1806.09543}.

\bibitem[MS10]{meyer_resolutions_2010}
R.~Meyer and M.~Solleveld, \textsl{ Resolutions for representations of
  reductive {$p$}-adic groups via their buildings},
\newblock J. Reine Angew. Math. \textbf{ 647}, 115--150 (2010).

\bibitem[SS97]{SchneiderStuhler}
P.~Schneider and U.~Stuhler, \textsl{ Representation theory and sheaves on the
  {B}ruhat-{T}its building},
\newblock Inst. Hautes \'{E}tudes Sci. Publ. Math. (85), 97--191 (1997).

\bibitem[Vig96]{vignerasmod}
M.-F. Vign\'eras,
\newblock \textsl{ Repr\'esentations {$l$}-modulaires d'un groupe r\'eductif
  {$p$}-adique avec {$l\ne p$}}, volume 137 of \textsl{ Progress in
  Mathematics},
\newblock Birkh\"auser Boston, Inc., Boston, MA, 1996.

\bibitem[Vig97]{vignsheves}
M.-F. Vign\'{e}ras, \textsl{ Cohomology of sheaves on the building and
  {$R$}-representations},
\newblock Invent. Math. \textbf{ 127}(2), 349--373 (1997).

\bibitem[Wan17]{Wanga}
H.~Wang, \textsl{ L'espace sym\'etrique de {D}rinfeld et correspondance de
  {L}anglands locale {II}},
\newblock Math. Ann. \textbf{ 369}(3-4), 1081--1130 (2017).

\end{thebibliography}
\end{document}